\documentclass[11pt, bibliography=totoc]{scrartcl}
\usepackage[T1]{fontenc}

\usepackage{amssymb,amsmath,amsthm, mathrsfs}

\usepackage{lmodern}
\usepackage{microtype}

\usepackage{mathtools} 
\mathtoolsset{showonlyrefs} 
\usepackage{nicefrac,xfrac}
\usepackage{tikz-cd}
\usepackage{enumitem}
\usepackage[colorlinks=true, urlcolor=blue, linkcolor=blue, citecolor=blue]{hyperref}

\usepackage[
  backend=biber,
  style=alphabetic,
  sorting=nyt,
  maxnames=6,
  maxalphanames=6,
  giveninits=true,
  doi=false,
  isbn=false,
  url=false,
]{biblatex}
\DeclareFieldFormat{extraalpha}{#1}
\DeclareLabelalphaTemplate{
  \labelelement{
    \field[final]{shorthand}
    \field{label}
    \field[strwidth=3,strside=left,ifnames=1]{labelname}
    \field[strwidth=1,strside=left]{labelname}
  }
}

\numberwithin{equation}{section}
\newtheorem{theorem}{Theorem}[section]
\newtheorem{proposition}[theorem]{Proposition}
\newtheorem{corollary}[theorem]{Corollary}

\newtheorem{lemma}[theorem]{Lemma}

\theoremstyle{definition}

\newtheorem{remark}[theorem]{Remark}

\theoremstyle{remark}
\newtheorem{example}[theorem]{Example}

\newcommand{\R}{\mathbb{R}}  
\newcommand{\C}{\mathbb{C}}  
\newcommand{\Sph}{\mathbb{S}}  
\DeclareMathAlphabet{\mathbbold}{U}{bbold}{m}{n}

\DeclareMathOperator{\id}{id}

\DeclareMathOperator{\rk}{rk}

\DeclareMathOperator{\Tr}{Tr}

\DeclareMathOperator{\supp}{supp}
\DeclareMathOperator{\sing}{sing}

\DeclareMathOperator{\Vol}{Vol}

\DeclarePairedDelimiterX\set[2]{\{}{\}}{\,#1 \;\delimsize\vert\; #2\,}
\renewcommand*{\d}{\mathrm{d}} 
\newcommand*{\abs}[1]{\left|#1\right|} 
\newcommand*{\norm}[1]{\left\|#1\right\|} 
\newcommand*{\brr}[1]{\left(#1\right)} 
\newcommand*{\brs}[1]{\left[#1\right]} 
\newcommand*{\brc}[1]{\left\{#1\right\}} 
\newcommand*{\brt}[1]{\left\langle #1\right\rangle } 
\mathchardef\mhyphen="2D 
\renewcommand*{\epsilon}{\varepsilon}
\renewcommand*{\phi}{\varphi}
\newcommand*{\harm}{\mathbb{H}} 

\makeatletter
\newcommand*{\oset}[3][0.45ex]{%
  \mathrel{\mathop{#3}\limits^{
    \vbox to#1{\kern-2\ex@
    \hbox{$\scriptstyle#2$}\vss}}}}
\makeatother

\def\XXint#1#2#3{{\setbox0=\hbox{$#1{#2#3}{\int}$ }
\vcenter{\hbox{$#2#3$ }}\kern-.6\wd0}}

\newcommand{\overbar}[1]{\mkern 1.5mu\overline{\mkern-1.5mu#1\mkern-1.5mu}\mkern 1.5mu}
\renewcommand*{\t}{\boldsymbol{\tau}} 
\newcommand*{\wH}[2][]{H^{1}\ifx\empty#1\empty(#2)\else(#1,#2)\fi} 
\newcommand*{\D}{\mathrm{Lip}}
\newcommand*{\n}{n}

\begin{document}
\author{Denis Vinokurov}
\date{}
\title{Eigenvalue optimization via a first-variation formula}
\maketitle
\begin{abstract}
  We compute the Clarke subdifferential of the $k$th eigenvalue functional on the space of self-adjoint operators,
  obtaining a first-variation formula that remains valid even when the eigenvalue lies at the edge of the essential spectrum.
  This formula provides an effective tool for describing the structure of critical points in eigenvalue optimization problems
  and can also yield simple proofs of the existence of optimizers.
  We illustrate these advantages through applications to the optimization of weighted Laplace and Steklov eigenvalues.
  In particular, we characterize all optimal weights, thereby answering some open questions posed by Kokarev,
  and give a short proof that such weights exist.
\end{abstract}
\tableofcontents

\section{Introduction and main results}

\subsection{Weighted optimization of Laplace and Steklov eigenvalues}

Let $\Omega \subset M$ be a bounded \emph{continuous} domain in a connected Riemannian manifold $(M,g)$ of dimension $d\geq 2$. For convenience, we assume that all our domains are \emph{connected}. Note that we also include the case
of $\partial \Omega = \varnothing$, that is, when $\Omega = M$ is itself a closed manifold.
By a continuous domain, we mean a domain that can be locally represented as the epigraph of a continuous function. In particular, such a domain can be viewed as a topological
manifold with boundary. Bounded continuous domains enjoy the following properties
(see, for example, \cite[Theorem~1.1.6/2]{Mazja:1985:-sobolev-spaces} and \cite[Theorem~V.4.17]{Edmunds-Evans:1987:spec-theory-diff-operators}):
\begin{itemize}
    \item $C^\infty(\overbar{\Omega})$ is dense in $H^1(\Omega)$;
    \item the embedding $H^1(\Omega) \to L^2(\Omega)$ is compact.
\end{itemize}
Let $\mu \in \mathcal{M}_+(\overbar{\Omega})$ be a positive Radon measure.
Following~\cite{Grigoryan-Netrusov-Yau:2004:eignval-of-ellipt-oper}, for every integer $k \geq 0$, we define $\lambda_k(\mu) \in [0,\infty]$ by
\begin{equation}\label{eq:meas-eigen}
    \lambda_k(\mu) = \lambda_k(\Omega,\mu) := \inf_{V_{k+1} \subset \D(\overbar{\Omega})} \sup_{\phi \in V_{k+1}\setminus\brc{0}} \frac{\int_{\Omega} \abs{\d \phi}^2}{\int_{\overbar{\Omega}}\phi^2 \d \mu}.
\end{equation}
One may equivalently use $C^\infty(\overbar{\Omega})$ or $H^1\cap C^0(\overbar{\Omega})$ in place of $\D(\overbar{\Omega})$.
Note that $\lambda_k(\mu) < \infty$ provided that $L^2(\overbar{\Omega}, \mu)$ is at least $(k+1)$-dimensional. In particular,
$\lambda_k(\mu) < \infty$ whenever $\mu \in \mathcal{M}_+^c(\overbar{\Omega})$ is continuous (that is, nonatomic). Moreover,
Proposition~\ref{prop:meas-bilinear} implies
\begin{equation}\label{def:capacitory-meas}
  \bigcup_{k\geq 1}\set{\mu \in \mathcal{M}_+^c(\overbar{\Omega}){}\setminus\brc{0}}{\lambda_k(\mu) > 0} \subset \mathfrak{Bil}[H^1(\Omega)];
\end{equation}
that is, if $\lambda_k(\mu) >0$ for some $k \geq 1$, then $\mu$ induces a bounded bilinear form
on $H^1(\Omega)$. The eigenvalues $\lambda_k(\mu)$ then constitute precisely the bottom part of the spectrum of the following weighted ``Neumann'' eigenvalue problem:
\begin{equation}
  \Delta u = \lambda u \mu \quad \text{in }\ H^1(\Omega)^*,
\end{equation}
where $\Delta = \d^* \d$ denotes the Laplace--Beltrami operator on $M$.

We set
\begin{equation}
    \overbar{\lambda}_k(\mu) = \mu(\overbar{\Omega}){\lambda}_k(\mu)
\end{equation}
and $\overbar{\lambda}_k(0) := 0$.
Now, choose a smooth $\Omega'$ such that $\Omega \Subset \Omega' \Subset M$ and extend $\mu$ to all of $M$ by setting $\mu(S) = \mu(\overbar{\Omega}\cap S)$.
By definition, $\lambda_k(\Omega,\mu) \leq \lambda_k(\Omega',\mu)$.
Viewing $\Omega$ as a subset of the double of $\Omega'$ and varying the metric outside $\Omega$,
we obtain from~\cite[Remark~5.10]{Grigoryan-Netrusov-Yau:2004:eignval-of-ellipt-oper} that
\begin{equation}
  \overbar{\lambda}_k(\Omega,\mu) \leq C\Vol_g(\Omega)^{1-2/d} k^{2/d} \quad \forall \mu \in \mathcal{M}_+^c(\overbar{\Omega}),
\end{equation}
where $C = C(\Omega',[g])$.
This motivates the following definitions:
\begin{equation}
  \Lambda_k(\Omega,g) = \sup_{\mu \in \mathcal{M}_+^c(\overbar{\Omega})} \overbar{\lambda}_k(\mu)
  \quad\text{and}\quad
  \Lambda_k^*(\Omega,g) =\sup_{\mu \in L^1_+(\Omega)} \overbar{\lambda}_k(\mu).
\end{equation}
The motivation for considering $\Lambda_k^*(\Omega,g)$ is that
$\Lambda_k^*(\Omega,g) = \sup \set{\overbar{\lambda}_k(\mu)}{\mu \in C^\infty_+(\overbar{\Omega})}$
when $\Omega$ is smooth
(see \cite[Proposition~2.9]{Vinokurov:2025:higher-dim-harm-eigenval}), and in dimension 2,
it coincides with the conformal optimization problem:
\begin{equation}
  \Lambda_k^*(\Omega,g) = \Lambda_k^*(\Omega,[g]) = \sup_{\tilde{g} \in [g]} \overbar{\lambda}_k(\tilde{g}).
\end{equation}

While it is not difficult to prove the existence of maximizing measures for $\Lambda_k(\Omega,g)$
(cf. \cite[Theorem~$B_1$]{Kokarev:2014:measure-eigenval} for closed surfaces), very little is
known about the regularity of such measures (cf. \cite[Theorem~$C_k$]{Kokarev:2014:measure-eigenval}).
On the other hand, proving the existence of maximizers for $\Lambda_k^*(\Omega,g)$ is considerably more involved \cite{Nadirashvili-Sire:2015:conf-spec-and-harm-maps,Petrides:2018:exist-of-max-eigenval-on-surfaces,
Karpukhin-Nadirashvili-Penskoi-Polterovich:2022:existence, Karpukhin-Stern:2024:new-character-of-conf-eigenval,Karpukhin-Stern:2024:harm-map-in-higher-dim,
Petrides:2024:conf-class-opt-lms, Vinokurov:2025:sym-eigen-val-lms, Vinokurov:2025:higher-dim-harm-eigenval}. However, by
\cite[Theorem~1.1]{Vinokurov:2026:p-harm-and-conf-class-opt}, the regularity of maximizing measures is well understood when $\partial \Omega = \varnothing$
(see also \cite{Karpukhin-Stern:2024:new-character-of-conf-eigenval,Karpukhin-Stern:2024:harm-map-in-higher-dim} for complementary results for ``admissible'' measures).
Namely, after rescaling, any maximizing $\mu \in L^1(\Omega)$ has the form $\mu = \abs{\d u}^2$ for a locally (spectrally) stable harmonic map $u \in H^1(\Omega, \Sph^n)$ for some $n \geq 1$, and,
in fact, $u \in C^\infty(\Omega, \Sph^n)$ in dimensions $2\leq d \leq 6$.

We denote the unit sphere and unit ball of $\R^\infty := \ell^2$ by $\Sph^{\infty}$ and $\mathbb{B}^\infty$. We also identify functions with measures via the Riemannian volume form.
As an application of Corollary~\ref{cor:subdiff-eigenval}, we can show that no new maximizers arise when passing from $\Lambda_k^*(\Omega,g)$ to $\Lambda_k(\Omega,g)$:
\begin{theorem}\label{thm:qualitative-stability}
  Let $\Omega \subset M$ be a bounded continuous domain in a Riemannian manifold $(M,g)$ of dimension $d \geq 2$, and let
  $\mu \in \mathcal{M}^c_+(\overbar{\Omega})$ be such that $\overbar{\lambda}_k(\mu) = \Lambda_k(\Omega,g)$ for some $k \geq 1$. Then there exists
  a harmonic map $u \in H^1(\Omega, \Sph^{\infty})$ whose components are $k$th eigenfunctions ($\Delta u = \lambda_k u \mu$)
  such that $\lambda_k \mu = \abs{\d u}^2$.
  Moreover, $u \in C^\infty(\Omega\setminus\sing u, \Sph^{\infty})$, where $\sing u \cap \Omega$ is a relatively closed set of Hausdorff
  dimension at most $d - 7$.

  In addition:
  \begin{itemize}
    \item if $\partial\Omega = \varnothing$, then $im\, u \subset \Sph^n$  for some $\Sph^n \subset \Sph^\infty$ and $\sing u = \varnothing$ for $2\leq d \leq 6$;
    \item if $H^1(\Omega)\to L^2(\mu)$ is compact and $\partial \Omega$ is smooth, then $u \in C^\infty(\overbar{\Omega}, \Sph^n)$  for some $\Sph^n \subset \Sph^\infty$.
  \end{itemize}
\end{theorem}
However, an example of eigenvalue optimization on $\Sph^d$ (see \cite{Vinokurov:2025:higher-dim-harm-eigenval}) shows that the upper bound $d-7$ is sharp
and $H^1(\Omega)\to L^2(\mu)$ need not be compact for maximizing measures.
\begin{remark}\label{rem:local-max}
  Inspecting the proof of Theorem~\ref{thm:qualitative-stability}, we see that the global maximizing assumption may be replaced by the following local condition:
  $\mu$ satisfies the condition in \eqref{def:capacitory-meas},
  and $0$ is a local maximizer of the functional $\rho \mapsto \overbar{\lambda}_k(\mu + \rho)$ on $L^\infty_+(\Omega)$.
  Thus, only compact perturbations of $\mu$ are needed.
\end{remark}
\begin{remark}
  Once a partial regularity of $u$ is established up to the boundary,
  one may actually hope that $im\, u \subset \Sph^n$ for some finite-dimensional sphere $\Sph^n$, as in \cite[Section~5]{Vinokurov:2025:higher-dim-harm-eigenval}.
\end{remark}
  Thus, Theorem~\ref{thm:qualitative-stability} states that every maximizing Radon measure is induced by a harmonic map into a sphere.
  This strengthens several previously known results. Under the additional assumption that $\mu \in L^1(\Omega)$,
  the corresponding statement was proved in \cite[Theorem~1.1]{Vinokurov:2026:p-harm-and-conf-class-opt};
  see also \cite[Remark~1.3]{Vinokurov:2026:p-harm-and-conf-class-opt} for the case $\mu \in (H^{1,d/(d-1)})^*$ when $d \geq 3$. Moreover,
  \cite{Karpukhin-Stern:2024:new-character-of-conf-eigenval,Karpukhin-Stern:2024:harm-map-in-higher-dim} established the result for
  $\lambda_2$ in dimension $d=2$ and for $\lambda_1$ in dimensions $2\le d\le 5$ assuming that the canonical map $H^1 \to L^2(\mu)$ is compact.
\begin{remark}
  While this manuscript was in the final stage of preparation, Ambrosio and Siclari \cite{Ambrosio-Siclari:2026:admissible-measures} proved,
  using different methods,
  the corresponding regularity statement for every locally $\overbar{\lambda}_k$-maximizing Radon measure $\mu$ for which the canonical map $H^1(\Omega)\to L^2(\mu)$ is compact. As one can see from Remark~\ref{rem:local-max}, our result therefore extends theirs to
  arbitrary (locally) $\overbar{\lambda}_k$-maximizing nonatomic Radon measures. Corollary~\ref{cor:nonessen-boundary-dim-2} similarly extends \cite[Corollary~1.4]{Ambrosio-Siclari:2026:admissible-measures}.
\end{remark}

From the theory of harmonic maps, the singular set of $\mu = \abs{\d u}^2$ coincides with that of $u$, which is given by $\sing u \cap \Omega = \set{p \in \Omega}{\liminf_{r \to 0} r^{2-d}\int_{B_r(p)}\abs{\d u}^2 > 0}$;
see \cite{Vinokurov:2025:higher-dim-harm-eigenval}.
The theorem above immediately answers the first two questions in \cite[Section~6.2]{Kokarev:2014:measure-eigenval}:
\begin{corollary}
  For every $d\geq 2$ and every $\overbar{\lambda}_k$-maximizing measure $\mu \in \mathcal{M}^c_+(\overbar{\Omega})$,
  the Hausdorff dimension of the singular set $\sing \mu \cap \Omega = \set{p \in \Omega}{\liminf_{r \to 0} r^{2-d}\mu(B_r(p)) > 0}$ is at most $d - 7$,
  and this bound is sharp. In particular, for $2 \leq d \leq 6$, $\mu$ is regular in the interior, that is, $\mu \in C^\infty(\Omega)$.
\end{corollary}

When the continuous domain $\Omega$ has $\mathcal{H}^{d-1}(\partial \Omega) < \infty$, we define
\begin{equation}
  \sigma_k(\Omega) = \lambda_k(\Omega, \mathcal{H}^{d-1}|_{\partial \Omega})
  \quad\text{and}\quad
  \overbar{\sigma}_k(\Omega) = \sigma_k(\Omega)\mathcal{H}^{d-1}(\partial \Omega),
\end{equation}
which is the usual variational definition of the Steklov eigenvalues.
Note that in the continuous context, the Steklov spectrum need not be discrete (see, for example, \cite{Nazarov-Taskinen:2020:noncomp-steklov}).
When $\Omega$ is Lipschitz and $\mu \in \mathcal{M}_+(\partial\Omega)$, then the (weak) eigenvalue
problem $\Delta u = \lambda u \mu$ is equivalent to
\begin{equation}\label{eq:Steklov-problem}
 \begin{cases}
   \Delta u|_{\Omega} = 0,
   \\ \partial_n u\left|_{\partial\Omega}\right. = \lambda u \mu,
 \end{cases}
\end{equation}
where $\partial_n$ denotes differentiation in the direction of the outward-pointing unit normal vector $\boldsymbol{\n}$ to $\partial \Omega$, and
$\partial_n u$ is defined in the weak sense ($\Delta u|_{\Omega} = 0 \in H^{1}_0(\Omega)^*$).
This motivates the following weighted Steklov optimization problems:
\begin{equation}
  \Sigma_k(\Omega,g) = \sup_{\mu \in \mathcal{M}_+^c(\partial\Omega)} \overbar{\lambda}_k(\mu)
  \quad\text{and}\quad
  \Sigma_k^*(\Omega,g) =\sup_{\mu \in L^1_+(\partial\Omega)} \overbar{\lambda}_k(\mu).
\end{equation}
As for $\Lambda_k^*$, when $\Omega$ is smooth, we have
$\Sigma_k^*(\Omega,g) = \sup \set{\overbar{\lambda}_k(\mu)}{\mu \in C^\infty_+(\partial \Omega)}$,
and in dimension 2,
\begin{equation}
  \Sigma_k^*(\Omega,g) = \Sigma_k^*(\Omega,[g]) = \sup_{\tilde{g} \in [g]} \overbar{\sigma}_k(\tilde{g}).
\end{equation}
The existence and regularity results for $\Sigma_k^*$ are well-known in dimension 2
\cite{Fraser-Schoen:2016:sharp-steklov-bounds,Petrides:2019:max-steklov-eigenval-on-surfaces,Petrides:2024:conf-class-opt-lms, Vinokurov:2025:sym-eigen-val-lms}, and the existence part extends
to higher dimensions (see \cite{Vinokurov:2025:higher-dim-harm-eigenval}). However, the existing proofs are as involved as those for $\Lambda_k^*$.

Restricting the admissible measures to those supported on $\partial \Omega$
yields the Steklov eigenvalue maximization problem.
We obtain a result analogous to Theorem~\ref{thm:qualitative-stability},
which allows one to work with $\Sigma_k$ instead of $\Sigma_k^*$.
\begin{theorem}\label{thm:qualitative-stability-stek}
  Let $\Omega \subset M$ be a Lipschitz domain in a Riemannian manifold $(M,g)$ of dimension $d \geq 2$, $\partial \Omega \neq \varnothing$, and let
  $\mu \in \mathcal{M}^c_+(\partial\Omega)$ be such that $\overbar{\lambda}_k(\mu) = \Sigma_k(\Omega,g)$ for some $k \geq 1$. Then there exists
  a free boundary harmonic map $u \in H^1(\Omega, \mathbb{B}^{\infty})$ whose components are the $k$th eigenfunctions ($\Delta u = \lambda_k u \mu$)
  such that $\lambda_k \mu = \abs{\partial_{\n}u}\in L^1(\partial\Omega)$.

  If $\Omega$ is smooth and either $d = 2$ or $H^1(\Omega)\to L^2(\mu)$ is compact,
  we have $u \in C^\infty(\overbar{\Omega}, \mathbb{B}^n)$ for some $\mathbb{B}^n \subset \mathbb{B}^\infty$.
\end{theorem}
Here, by a (weakly) free boundary harmonic map into $\mathbb{B}^{\infty}$, we mean a map $u \in H^1(\Omega, \mathbb{B}^{\infty})$ satisfying
$u(\partial \Omega) \subset \mathbb{S}^\infty$ and
\begin{equation}\label{eq:free-bndry-ball}
  \int_{\Omega} \brt{\d u, \d v} = 0 \quad\forall v \in L^\infty \cap H^1(\Omega, \ell^2) \text{ s.t. } v(x) \in T_{u(x)}\mathbb{S}^\infty \ \text{for a.e.}\ x \in \partial\Omega.
\end{equation}
\begin{remark}[On the regularity of maximizing harmonic maps]
  The regularity conclusions of Theorems~\ref{thm:qualitative-stability} and~\ref{thm:qualitative-stability-stek}
  can most likely be strengthened up to the boundary of $\Omega$ (at least when $\Omega$ is smooth), as the following reasoning suggests.
  For each pair of manifolds $S \subset N$, we can define
  (weakly) harmonic maps $u \in H^1(\Omega, N)$ \emph{with respect to the free boundary condition} $u(\partial \Omega) \subset S$ as critical points
  of the energy functional on the class of maps $\tilde{u}$ satisfying $\tilde{u}(\partial \Omega) \subset S$ (see, for example, \cite{Duzaar-Steffen:1989:reg-free-bndry-harm-maps,Scheven:2006:regularity-free-bndry-harm-maps}).
  One checks that maximizing maps in Theorems~\ref{thm:qualitative-stability}
  and~\ref{thm:qualitative-stability-stek}
  are free boundary harmonic maps with $(S,N) = (\Sph^\infty, \Sph^\infty)$ and $(S,N) = (\Sph^\infty, \mathbb{R}^\infty)$, respectively.

  In both cases, the maps are locally stable and therefore locally energy-minimizing (as in \cite[Lemma~5.5]{Vinokurov:2025:higher-dim-harm-eigenval}).
  This suggests that the regularity theory for free boundary harmonic maps into $\Sph^\infty$ and $\mathbb{R}^\infty = \ell^2$
  should coincide with the corresponding theory for finite-dimensional targets
  $\Sph^n$ and $\mathbb{R}^n$ (by analogy with \cite[Section~5]{Vinokurov:2025:higher-dim-harm-eigenval} for $\partial \Omega = \varnothing$).
  On the other hand, we know (for example, from \cite{Duzaar-Steffen:1989:reg-free-bndry-harm-maps}) that for energy-minimizing
  free boundary harmonic maps into finite-dimensional targets,
  the Hausdorff dimension of $\sing u$ is at most $d-3$. Furthermore (see \cite[Remark~4.3]{Duzaar-Steffen:1989:reg-free-bndry-harm-maps}
  and \cite{Hong-Wang:1999:stable-harmonic-maps}),
  even stronger regularity results may be obtained if one proves the nonexistence of nontrivial stable free boundary harmonic maps in
  $C^\infty(\overbar{\Sph^{m}_+}, \Sph^\infty)$ and $C^\infty(\overbar{\Sph^{m}_+}, \mathbb{B}^\infty)$, respectively.
\end{remark}

\begin{remark}
  As in Remark~\ref{rem:local-max}, the global maximizing assumption in Theorem~\ref{thm:qualitative-stability-stek}
  may be weakened to local maximality. For the Steklov problem, we restrict our attention to Lipschitz domains rather than merely continuous ones, as we need to
  ensure the existence of at least one measure
  $\nu \in \mathcal{M}_+^c(\partial \Omega)$ such that $\supp \nu = \partial \Omega$ and the canonical map $H^1(\Omega) \to L^2(\nu)$ is compact;
  for example, $\nu = \mathcal{H}^{d-1}|_{\partial \Omega}$.

  We also use the Lipschitz condition to prove that $\partial_n u \in L^1(\partial \Omega, \ell^2)$. In fact, a similar
  argument can be used (see Remark~\ref{rem:free-bndry-l1}) to show that $\partial_n u \in L^1$ for any weakly harmonic map
  $u \in H^1(\Omega, N)$ with respect to the free boundary condition $u(\partial \Omega) \subset S$.
\end{remark}

\subsection{Existence of \texorpdfstring{$\overbar{\lambda}_k$}{Lg}-maximizing harmonic maps
  via the weak-\texorpdfstring{$*$}{Lg} stability of maximizing measures}
  \label{sec:weak-stability}

In this section, we present a simple scheme for proving the existence of maximizers in appropriate eigenvalue
optimization problems. In the context of $\Lambda_k^*$ and $\Sigma_k^*$, it yields a substantially shorter proof than the existing approaches.

To prove the existence of harmonic
maps realizing $\Lambda_k^*(\Omega,g)$, existing approaches \cite{Nadirashvili-Sire:2015:conf-spec-and-harm-maps,Petrides:2018:exist-of-max-eigenval-on-surfaces,
Karpukhin-Nadirashvili-Penskoi-Polterovich:2022:existence, Karpukhin-Stern:2024:new-character-of-conf-eigenval,Karpukhin-Stern:2024:harm-map-in-higher-dim,
Petrides:2024:conf-class-opt-lms, Vinokurov:2025:sym-eigen-val-lms, Vinokurov:2025:higher-dim-harm-eigenval}
begin with an appropriately chosen maximizing sequence and then establish convergence with sufficient regularity to ensure that the limit belongs to the desired class.

Instead of analyzing maximizing subsequences, we follow the idea in Remark~\ref{rem:abstract-min-idea} and analyze their weak-$*$ limits or, to be more precise,
the maximizers of the relaxed problem $\Lambda_k(\Omega,g)$. We do so with the following two observations in mind:
\begin{itemize}
  \item $\mu \mapsto \overbar{\lambda}_k(\mu)$
  is upper semicontinuous with respect to the weak-$*$ convergence of measures;
  \item under certain conditions on $\Lambda_k^*(\Omega,g)$, a bubbling analysis along a maximizing sequence (see, for example, \cite{Vinokurov:2025:sym-eigen-val-lms,Vinokurov:2025:higher-dim-harm-eigenval}) shows that
  if $\mu_n \in \mathcal{M}_+^c(\overbar{\Omega})$, $\mu_n(\overbar{\Omega}) = 1$, and $\overbar{\lambda}_k(\mu_n) \to \Lambda_k^*(\Omega,g)$,
  then, up to a subsequence,  $\mu_n \oset{w^*}{\to} \mu \in \mathcal{M}_+^c(\overbar{\Omega})$, that is, the limit $\mu$ is
  also continuous, and hence $\overbar{\lambda}_k(\mu) = \Lambda_k^*(\Omega,g)$.
\end{itemize}
Therefore, maximizing measures enjoy a weak-$*$ stability property:
\emph{every maximizing sequence admits a subsequence converging weak-$*$ to a maximizer}. More generally,
the superlevel sets of $\lambda_k$ are weak-$*$ stable (see Remark~\ref{rem:superlevel-stability}).
One then shows that $\mu \in \mathfrak{Bil}[H^1(\Omega)]$, so one can apply Corollary~\ref{cor:subdiff-eigenval}
to compute $\partial_c \overbar{\lambda}_k(\mu + 0)$ and thereby obtain the existence of a harmonic map inducing $\mu$
(Theorems~\ref{thm:qualitative-stability} and~\ref{thm:qualitative-stability-stek}). Set $E[u] = \int_{\Omega} \abs{\d u}^2$.
\begin{corollary}\label{cor:existence}
  Let $\Omega \subset M$ be a bounded continuous domain in a Riemannian manifold $(M,g)$ of dimension $d \geq 2$. Then $\Lambda_k(\Omega,g) = \Lambda_k^*(\Omega,g)$.
  Moreover, suppose that one of the following conditions is satisfied:
  \begin{itemize}
    \item $d \geq 3$,
    \item $d = 2$ and $\Lambda_k(\Omega,[g]) > \Lambda_{k-1}(\Omega,[g]) + 8\pi$.
  \end{itemize}
  If $\brc{\mu_n} \subset \mathcal{M}^c_+(\overbar{\Omega})$ is a sequence of probability measures such that $\overbar{\lambda}_k(\mu_n) \to \Lambda_k(\Omega,g)$,
  then there exists a probability measure $\mu = \frac{1}{E[u]}\abs{\d u}^2 \in L^1_+(\Omega)$, induced by a map $u$ as in Theorem~\ref{thm:qualitative-stability},
  such that $\overbar{\lambda}_k(\mu) = \Lambda_k(\Omega,g)$ and, up to a subsequence, $\mu_n \oset{w^*}{\to} \mu$.
\end{corollary}
See also \cite{Karpukhin-Nahon-Polterovich-Stern:2021:stability} for complementary stability results, including the case of $\Sph^2$.
\begin{remark}\label{rem:lambda-1}
  In dimension $2$, Petrides \cite{Petrides:2014:heat-kernel} proved that one always has $\Lambda_1(\Omega, [g]) > 8\pi$ for a closed surface
  $\Omega$ that is not homeomorphic to $\Sph^2$. Therefore, $\Lambda_1(\Omega, [g])$ is always realized by a harmonic map to a sphere
  (even when $\Omega \approx \Sph^2$, in which case the realizing map may be taken to be $\id_{\Sph^2}$).
\end{remark}
\begin{corollary}\label{cor:existence-stek}
  Let $\Omega \subset M$ be a bounded Lipschitz domain in a Riemannian manifold $(M,g)$ of dimension $d \geq 2$ and $\partial \Omega \neq \varnothing$. Then $\Sigma_k(\Omega,g) = \Sigma_k^*(\Omega,g)$.
  Moreover, suppose that one of the following conditions is satisfied:
  \begin{itemize}
    \item $d \geq 3$,
    \item $d = 2$ and $\Sigma_k(\Omega,[g]) > \Sigma_{k-1}(\Omega,[g]) + 2\pi$.
  \end{itemize}
   If $\brc{\mu_m} \subset \mathcal{M}^c_+(\partial\Omega)$ is a sequence of probability measures such that $\overbar{\lambda}_k(\mu_m) \to \Sigma_k(\Omega,g)$,
  then there exists a probability measure $\mu = \frac{1}{E[u]}\abs{\partial_n u} \in L^1_+(\partial\Omega)$, induced by a map $u$ as in Theorem~\ref{thm:qualitative-stability-stek},
  such that $\overbar{\lambda}_k(\mu) = \Sigma_k(\Omega,g)$ and, up to a subsequence, $\mu_m \oset{w^*}{\to} \mu$.
\end{corollary}
\begin{remark}\label{rem:lower-lambda-bounds}
  Since $\Sigma_k(\Omega,g) = \Sigma_k^*(\Omega,g)$ and $\Lambda_k(\Omega,g) = \Lambda_k^*(\Omega,g)$,
  it follows from the known facts about $\Lambda_k^*$ and  $\Sigma_k^*$ (see also the proofs of the corollaries) that for $d=2$,
  \begin{equation}
    \Lambda_k(\Sph^2) = 8\pi k, \quad \Lambda_k(\Omega,[g]) \geq \Lambda_{k-1}(\Omega,[g]) + 8\pi
  \end{equation}
  and
   \begin{equation}
   \Sigma_k(\mathbb{D}^2) = 2\pi k, \quad \Sigma_k(\Omega,[g]) \geq \Sigma_{k-1}(\Omega,[g]) + 2\pi
  \end{equation}
  for all $k \geq 1$.
\end{remark}
In dimension $2$, every bounded continuous domain $\Omega \subset M$ is conformal to $\Omega^*\setminus\bigcup_{i=1}^s \mathbb{D}_i^2$ for some closed surface $\Omega^*$,
obtained by smoothly gluing disks to a neighborhood of $\partial \Omega \subset \overbar{\Omega}$ and uniquely determined up to homeomorphism. We can therefore always assume that the
conformal class $[g]$ on $\Omega$ is the restriction of a conformal class on $\Omega^*$.
Since $\Lambda_k(\Omega, g) = \Lambda_k(\Omega, [g])$ depends only on the conformal class $[g]|_{\Omega}$, it is natural to compare $\Lambda_k(\Omega, [g])$
and $\Lambda_k(\Omega^*, [g])$:
\begin{corollary}\label{cor:nonessen-boundary-dim-2}
  Let $\Omega$ be a bounded continuous domain in a compact connected Riemannian manifold $(M, g)$ and $k \geq 1$.
  \begin{itemize}
    \item One has $\overbar{\lambda}_k(\Omega, \mu) < \Lambda_k(\Omega, g)$ for all $\mu \in \mathcal{M}_+^c(\overbar{\Omega})$
  with $\Omega \setminus \supp \mu \neq \varnothing$.
    \item In particular, if $\Omega \neq M$, then $\overbar{\lambda}_k(\Omega, \mu) < \Lambda_k(M, g)$ for all $\mu \in \mathcal{M}_+^c(\overbar{\Omega})$.
  \end{itemize}
  Furthermore, let $\dim \Omega = 2$,  $\partial \Omega \neq \varnothing$, and $\Omega \approx \Omega^*\setminus\bigcup_{i=1}^s \mathbb{D}_i^2$, where $(\Omega^*, [g])$ is a
  closed surface as above.
  \begin{itemize}
    \item If $\Omega^* \approx \Sph^2$, then $\Lambda_k(\Omega, [g]) = \Lambda_k(\Omega^*, [g]) = 8\pi k$.
    \item If $\Omega^* \not\approx \Sph^2$, then $\Lambda_k(\Omega, [g]) < \Lambda_k(\Omega^*, [g])$.
  \end{itemize}
\end{corollary}
Taking $\supp \mu \subset \partial \Omega$, we obtain the strict inequality between weighted Steklov eigenvalues and $\Lambda_k(\Omega, g)$.

\subsection{Abstract differentiation of variational eigenvalues}
The preceding results follow from the first-variation formula developed in this section.
Let $H$ be a complex Hilbert space and $\mathfrak{L}_{sa}[H]$ be the space of bounded self-adjoint operators on $H$. For an operator $S \in \mathfrak{L}_{sa}[H]$, we define
its $k$th variational eigenvalue as
\begin{equation}
  \lambda_k(S) = \sup_{\dim V = k}
    \inf_{x \in V\setminus\brc{0}}
    \frac{\brt{Sx,x}}{\norm{x}^2},
\end{equation}
where $k \geq 1$. Note that this differs from Laplace/Steklov convention, where we have $k \geq 0$.
The eigenvalues $\brc{\lambda_k(S)}_{k \geq 1}$ constitute precisely the upper part of the spectrum $\sigma(S)$, and one of the following two
alternatives occurs:
\begin{equation}
  \brc{\lambda_k(S)}_{k=1}^\infty = \sigma(S)\cap (\sigma_*,\infty) \quad\text{and}\quad \lambda_k(S) \searrow  \sigma_*,
\end{equation}
where $\sigma_*=\sup \sigma_{ess}(S)$, or
\begin{equation}
  \brc{\lambda_k(S)}_{k=1}^\infty = \sigma(S)\cap [\sigma_*,\infty)\quad\text{with}\quad \lambda_{k_0}(S) = \lambda_{k_0+1}(S)=\cdots = \sigma_*
\end{equation}
for some $k_0$.

Recall that the dual of the space of (self-adjoint) compact operators $\mathfrak{K}_{sa}[H]$ is the space of (self-adjoint) trace-class operators $\mathfrak{N}_{sa}[H]$. Furthermore,
$\mathfrak{N}_{sa}[H]^* = \mathfrak{L}_{sa}[H]$. Trace-class operators are also known as nuclear operators in the theory of operator ideals on Banach spaces,
and this viewpoint appears useful in eigenvalue optimization (see, for example, \cite{Vinokurov:2026:p-harm-and-conf-class-opt}).

Consider an abstract eigenvalue optimization problem on a bounded subset of positive compact operators (so that $\sigma_* \equiv 0$):
\begin{equation}\label{abstract-eigenval-opt-problem}
  \lambda_k(S) \to \min, \quad S \in \mathcal{S} \subset \mathfrak{K}_+[H].
\end{equation}
The $k$th eigenvalue functional is Lipschitz on $\mathfrak{K}_+[H]$ (and even on $\mathfrak{L}_{sa}[H]$). Therefore, if $S \in \mathcal{S}$ is a local
minimizer of $\lambda_k$, one can study its properties using the Clarke subdifferential and the method of Clarke multipliers;
see, for example, \cite[Theorem~1]{Clarke:1976:lagrange-multipliers} and Section~\ref{subsec:clarke-subdif}.
Unfortunately, the minimum in problem~\eqref{abstract-eigenval-opt-problem} may not exist in the class of compact operators, since
bounded subsets of $\mathfrak{K}[H]$ are not necessarily compact (even in the weak topology). Eigenvalue optimization
on high-dimensional spheres (see \cite[Theorem~1.8]{Vinokurov:2025:higher-dim-harm-eigenval}) provides an example of this phenomenon.

On the other hand, the functional $S \mapsto \lambda_k(S)$ is lower semicontinuous with respect to the weak operator topology.
Thus, minimizers always exist in the weak-$*$ closure $\overline{\mathcal{S}}^{w^*} \subset \mathfrak{L}_{sa}[H] = \mathfrak{K}_{sa}[H]^{**}$.
Therefore, the following theorem shows that the method of Clarke multipliers remains applicable, even when $S \in  \mathfrak{L}_{sa}[H] \setminus \mathfrak{K}_{sa}[H]$.
For convenience, set $\lambda_{0}(S) := \infty$.
\begin{theorem}\label{thm:subdiff-calc-all}
  Let $S \in \mathfrak{L}_{sa}[H]$, and let its $k$th variational eigenvalue satisfy $\lambda_k(S) < \lambda_{k-1}(S)$.
  Then the Clarke subdifferential ${\partial_c \lambda_k(S)} \subset \mathfrak{L}_{sa}[H]^*$ of the Lipschitz
  function $\lambda_k \colon \mathfrak{L}_{sa}[H] \to \R$ at $S$ is given by
  \begin{equation}\label{generalized-clark-subdif-all}
    \begin{aligned}
      \partial_c \lambda_k(S) &= \set*{\phi \in \mathfrak{L}_{sa}[H]^*}{\norm{\phi} = 1,\ \phi \geq 0,\ \supp \phi \subset E_{(\lambda_k - \epsilon, \lambda_k]}(S)\ \forall \epsilon > 0}
      \\                      &= \bigcap_{\epsilon > 0} \overline{\operatorname{co}}^{w^*}\set*{x\otimes x}{x \in E_{(\lambda_k - \epsilon, \lambda_k]}(S),\ \norm{x} = 1}.
    \end{aligned}
  \end{equation}
\end{theorem}
  Here, $E_{(\lambda_k-\epsilon, \lambda_k]}(S)$ is the image of the spectral projection $\chi_{(\lambda_k-\epsilon, \lambda_k]}(S)$,
  and $\overbar{\operatorname{co}}^{w^*} A$ stands for the weak-$*$ closure of the convex hull of $A$.

Theorem~\ref{thm:subdiff-calc-all} generalizes the well-known subdifferential formulas for eigenvalues of compact
operators (see, for example, \cite{Borwein-Read-Lewis-Zhu:2000:subdif-of-spec-func}) or for eigenvalues of finite multiplicity (see Proposition~\ref{prop:subdiff-calc}).
Sometimes, we deal with a parametric eigenvalue optimization problem: for a function $f \colon E \to \mathfrak{L}_{sa}[H]$, one considers
\begin{equation}
  (\lambda_k \circ f)(p) \to \min, \quad p \in \mathcal{P} \subset E.
\end{equation}
Combining the chain rule (Proposition~\ref{prop:clarke-compose}) and Theorem~\ref{thm:subdiff-calc-all}
with standard properties of continuous maps, closures, and intersections, yields the following
\begin{corollary}\label{cor:subdiff-composed-eigenval}
  Let $E$ be a Banach space, and let $f \colon E \to \mathfrak{L}_{sa}[H]$ be a $C^1$ map defined on a neighborhood of $p \in E$.
  If $\lambda_k(f(p)) < \lambda_{k-1}(f(p))$, then
  $\partial_c(\lambda_k \circ f)(p) \subset E^*$ satisfies
  \begin{equation}
  \partial_c(\lambda_k \circ f)(p) \subset \bigcap_{\epsilon > 0} \overline{\operatorname{co}}^{w^*}
    \set*{\d_p f^* (x \otimes x) }{x \in E_{(\lambda_k-\epsilon, \lambda_k]}(f(p)),\ \norm{x} = 1}.
\end{equation}
\end{corollary}
  See \eqref{eq:robin-subdiff}, \eqref{eq:laplace-subdiff}, and \eqref{eq:steklov-subdiff} for some explicit
computations of $\partial_c(\lambda_k \circ f)(p)$. Observe that in these computations, we consider only
compact perturbations of the original operator, which appear to be precisely what is needed in applications:
\begin{corollary}\label{cor:subdiff-calc}
  In the context of Theorem~\ref{thm:subdiff-calc-all}, consider the function $K \mapsto {\lambda_k(S+K)}$,
  where $K \in \mathfrak{K}_{sa}[H]$. Then its Clarke subdifferential ${\partial_c \lambda_k(S + 0)} \subset \mathfrak{N}_{sa}[H]$
  satisfies the inclusion
  \begin{equation}\label{generalized-clark-subdif}
    \partial_c \lambda_k(S + 0) \subset
    \set*{\sum_{i=1}^\infty x_i \otimes x_i }{
      x_i \in E_{\lambda_k}(S),\ \sum_{i}  \norm{x_i}^2  \leq 1
    }.
  \end{equation}
  That is, $\partial_c \lambda_k(S + 0)$ is contained in the unit ball of positive
  trace-class operators supported on $E_{\lambda_k}$.
\end{corollary}
Note that if $\lambda_k(S) \in \sigma_{ess}(S)$, then $K = 0$ is a local minimum for $K \mapsto \lambda_k(S+K)$
and $0 \in \partial_c \lambda_k(S + 0)$, since compact perturbations do not change the essential spectrum.
\begin{remark}[On proving the existence of global minimizers]\label{rem:abstract-min-idea}
  The application of this corollary, presented in Section~\ref{sec:weak-stability}, can be described as follows.
  Suppose we are looking for minimizers of the problem \eqref{abstract-eigenval-opt-problem}. As noted above, we know that there always exists
  a global minimizer of the relaxed problem
  \begin{equation}\label{abstract-eigenval-opt-problem-weak}
    \lambda_k(S) \to \min, \quad S \in \overline{\mathcal{S}}^{w^*} \subset \mathfrak{L}_{+}[H].
  \end{equation}
  If $S_0 \in \overline{\mathcal{S}}^{w^*}$ is such a minimizing operator, the Clarke multiplier rule implies that
  \begin{equation}\label{eq:clarke-multipliers}
    0 \in \partial_c \lambda_k (S_0) + N_\mathcal{K}(S_0),
  \end{equation}
  where $\mathcal{K} = (S_0 + \mathfrak{K}_{sa}[H]) \cap \overline{\mathcal{S}}^{w^*}$ and
  $N_\mathcal{K}(S_0)$ is the Clarke normal cone of $\mathcal{K}$ at $S_0$ (see \cite{Clarke:1976:lagrange-multipliers}).
  If we are able to prove that \eqref{eq:clarke-multipliers} implies $S_0 \in \mathcal{S}$, then
  $S_0$ is a global minimizer for the original problem \eqref{abstract-eigenval-opt-problem} as well.
\end{remark}

Let $Herm\,[H], Herm_c\,[H] \subset \mathfrak{Bil}[H,\overbar{H}]$ be the subspaces of bounded and compact Hermitian bilinear forms on $H\times\overbar{H}$, respectively, where
$\overbar{H}$ denotes the conjugate Hilbert space of $H$. Note that
\begin{equation}
  Herm\,[H] \approx \mathfrak{L}_{sa}[H],\quad Herm_c\,[H] \approx \mathfrak{K}_{sa}[H],\quad
  \text{and}\quad \mathfrak{Bil}[H,\overbar{H}] \approx \mathfrak{L}[H],
\end{equation}
although these isomorphisms depend on the choice of inner product on $H$. The canonical duality identities are
\begin{equation}
  \mathfrak{Bil}_c[H,\overbar{H}]^* = H \widehat{\otimes}_\pi \overbar{H},\quad (H \widehat{\otimes}_\pi \overbar{H})^*
 = \mathfrak{Bil}[H,\overbar{H}]
\end{equation}
and
\begin{equation}
  Herm_c\,[H]^*  = (H \widehat{\otimes}_\pi \overbar{H})_{sym},
 \quad (H \widehat{\otimes}_\pi \overbar{H})_{sym}^* = Herm\,[H],
\end{equation}
where $\widehat{\otimes}_\pi$ denotes the projective tensor product
 (see Section~\ref{sec:proj-tens-prod}) and the subscript $sym$ denotes the subspace of symmetric tensors.

For $\mathfrak{a}, \mathfrak{b} \in Herm\,[H]$ such that $\mathfrak{b} \geq 0$ and $\ker \mathfrak{b} \cap \ker \mathfrak{a} = \brc{0}$, we
define variational eigenvalues $\lambda_k(\mathfrak{a},\mathfrak{b}) \in [-\infty, \infty]$ as follows:
\begin{equation}\label{eq:eigenval-quadratic}
  \lambda_k(\mathfrak{a},\mathfrak{b}) = \inf_{V_{k} \subset D} \sup_{x \in V_{k}\setminus\brc{0}} \frac{\mathfrak{a}[x]}{\mathfrak{b}[x]},
\end{equation}
where $V_k$ ranges over all $k$-dimensional subspaces of a dense subspace $D \subset H$, and we identify each bilinear form with its associated quadratic form, writing $\mathfrak{a}[x]:= \mathfrak{a}[x,x]$.
Recall that bilinear forms on $H \times \overbar{H}$ may be identified with bounded linear operators $\mathfrak{a},\mathfrak{b} \colon H \to \overbar{H}^* \approx H$,
where the isomorphism is given by the inner product.

In applications, we often deal with eigenvalues of two quadratic forms~\eqref{eq:eigenval-quadratic},
and the associated operator often has no convenient explicit representation. Hence, the following version of Theorem~\ref{thm:subdiff-calc-all} for quadratic forms is also of independent interest.
\begin{theorem}\label{thm:subdiff-eigenval-all}
    Let $\mathfrak{a}, \mathfrak{b} \in Herm\,[H]$ be such that
    \begin{itemize}
      \item $\sigma_{ess}(\mathfrak{a}) \subset (0,\infty)$,
      \item $\mathfrak{b} \geq 0$, and $\ker \mathfrak{b} \cap \set{x\in H}{\mathfrak{a}[x] \leq 0} = \brc{0}$.
    \end{itemize}
    Then $\mathfrak{a} + c \mathfrak{b} \geq \frac{1}{c}\norm{\cdot}^2$ and
    $(\lambda_k(\mathfrak{a},\mathfrak{b}) + c)^{-1} = \lambda_k(T)$ for $T = (\mathfrak{a} + c\mathfrak{b})^{-1}\mathfrak{b} \in \mathfrak{L}_{sa}[H]$ and some $c > 0$,
    where $H$ is equipped with the inner product $\mathfrak{a} + c\mathfrak{b}$.
    Furthermore, if $\lambda_{k-1}(\mathfrak{a},\mathfrak{b}) < \lambda_k(\mathfrak{a},\mathfrak{b})$ and $\rk \mathfrak{b} \geq k$, then
    the function $(\mathfrak{a}', \mathfrak{b}') \mapsto -\lambda_k(\mathfrak{a}', \mathfrak{b}')$ has a Lipschitz extension to a neighborhood of
    $(\mathfrak{a}, \mathfrak{b}) \in \brr{Herm\,[H]}^{\times 2}$, and its Clarke subdifferential $\partial_c (-\lambda_k)(\mathfrak{a}, \mathfrak{b}) \subset \brr{Herm\,[H]^*}^{\times 2}$
    satisfies the inclusion
    \begin{equation}
        \partial_c (-\lambda_k)(\mathfrak{a}, \mathfrak{b}) \subset
        \bigcap_{\epsilon > 0} \overline{\operatorname{co}}^{w^*}
          \set*{(- x\otimes x, \lambda_k  x\otimes x)}{x\in E_{[\lambda_k, \lambda_k + \epsilon)}(\mathfrak{a}, \mathfrak{b}),\ \mathfrak{b}[x] = 1},
    \end{equation}
    where $\lambda_k = \lambda_k(\mathfrak{a},\mathfrak{b})$ and $E_{I}(\mathfrak{a}, \mathfrak{b}) := E_{(I + c)^{-1}}(T)$.
\end{theorem}
\begin{remark}\label{rem:positive-ess-spec}
  The condition $\sigma_{ess}(\mathfrak{a}) \subset (0,\infty)$ is equivalent to the existence of a form $\mathfrak{r} \in Herm_c\,[H]$ such that $\mathfrak{r} + \mathfrak{a} \geq \epsilon \norm{\cdot}^2$ for some $\epsilon > 0$,
  as one can verify by using the spectral decomposition of $\mathfrak{a} = \mathfrak{a}_+ - \mathfrak{a}_-$. In particular, there is an orthogonal decomposition
  $H = E_{\sigma(\mathfrak{a}) > 0} \oplus E_{\sigma(\mathfrak{a}) \leq 0}$, and $\ker \mathfrak{a}_+ = E_{\sigma(\mathfrak{a}) \leq 0}$ is finite dimensional.
\end{remark}
Restricting to compact perturbations again yields
\begin{corollary}\label{cor:subdiff-eigenval}
    In the context of Theorem~\ref{thm:subdiff-eigenval-all},
    the function $(\mathfrak{s}_1, \mathfrak{s}_2) \mapsto -\lambda_k(\mathfrak{a}+\mathfrak{s}_1, \mathfrak{b}+ \mathfrak{s}_2)$ has a Lipschitz extension to a neighborhood of
    $(0, 0) \in \brr{Herm_{c}\,[H]}^{\times 2}$, and its Clarke subdifferential $\partial_c (-\lambda_k)(\mathfrak{a} + 0, \mathfrak{b} + 0) \subset (H \widehat{\otimes}_\pi \overbar{H})_{sym}^{\times 2}$
    satisfies the inclusion
    \begin{equation}
        \partial_c (-\lambda_k)(\mathfrak{a} + 0, \mathfrak{b} + 0) \subset
        \set*{(- \t, \lambda_k \t)}{\t = \sum_{i=1}^\infty x_i\otimes x_i,\  x_i \in E_{\lambda_k}(\mathfrak{a}, \mathfrak{b}),\ \mathfrak{b}[\t] \leq 1},
    \end{equation}
    where $\mathfrak{b}[x\otimes y] := \mathfrak{b}[x, y]$.
\end{corollary}
\begin{example}
  Consider a bounded Lipschitz domain $\Omega$ in a Riemannian manifold $(M,g)$. Let
  $\mu$ be a positive measure and $\nu$ be a signed measure such that $\mu,\nu \in \mathfrak{Bil}[H^1(\Omega)]$,
  $H^1(\Omega) \to L^2(\overbar{\Omega}, \nu_{-})$ is compact, where $\nu = \nu_+ - \nu_-$, and
  $\int_{\Omega} \abs{\d \phi}^2 + \int_{\overbar{\Omega}}\phi^2 \d \nu > 0$ for all $\phi \in H^1\setminus \brc{0}$ with $\int_{\overbar{\Omega}}\phi^2 \d \mu = 0$. Then Theorem~\ref{thm:subdiff-eigenval-all}
  and Corollary~\ref{cor:subdiff-eigenval} apply to the following variational eigenvalue problem:
  \begin{equation}
    \lambda_k(\nu,\mu) = \inf_{V_{k} \subset \D(\overbar{\Omega})} \sup_{\phi \in V_{k}\setminus\brc{0}} \frac{\int_{\Omega} \abs{\d \phi}^2 + \int_{\overbar{\Omega}}\phi^2 \d \nu}{\int_{\overbar{\Omega}}\phi^2 \d \mu}.
  \end{equation}
  In particular, this includes a weighted Robin eigenvalue problem
  \begin{equation}\label{eq:Robin-problem}
    \begin{cases}
      \Delta u|_{\Omega} = \lambda u \mu
      \\ \partial_n u\left|_{\partial\Omega}\right.  + u \nu = 0
    \end{cases}
  \end{equation}
  if we require $\supp \nu \subset \partial \Omega$ and $\mu(\partial\Omega) = 0$. In that case,
  Corollary~\ref{cor:subdiff-eigenval} and Proposition~\ref{prop:clarke-compose} imply that the subdifferential of the function $(\rho, \varrho)\mapsto -\lambda_k(\nu + \rho,\mu + \varrho)$,
  where $(\rho, \varrho) \in L^\infty(\partial \Omega)\times L^\infty(\Omega)$, satisfies
  \begin{equation}\label{eq:robin-subdiff}
    \begin{multlined}
      \partial_c(-\lambda_k)(\nu + 0,\mu + 0) \\ \subset
      \set*{\brr{- \abs{u}^2\big|_{\partial\Omega}, \lambda_k \abs{u}^2}}{u \in H^1(\Omega,\ell^2) \text{ solving \eqref{eq:Robin-problem} with } \lambda = \lambda_k,\ \int_{\Omega} \abs{u}^2\d\mu \leq 1}.
    \end{multlined}
  \end{equation}
  See also Section~\ref{sec:meas-as-forms} for a criterion ensuring that $\mu \in \mathfrak{Bil}[H^1(\Omega)]$.
\end{example}

\subsection{Discussion of an extension to \texorpdfstring{$p$}{Lg}-harmonic map optimization}
  As noted above, Theorem~\ref{thm:qualitative-stability} is an improvement of \cite[Theorem~1.1]{Vinokurov:2026:p-harm-and-conf-class-opt}
  for $\overbar{\lambda}_k(\mu)$-maximization; that theorem additionally assumes that $\mu \in L^1_+(\Omega)$. However, \cite[Theorem~1.1]{Vinokurov:2026:p-harm-and-conf-class-opt}
  also applies to the optimization of $\overbar{\lambda}_{k,p}(\nu,\mu)$, where $\nu \in L_+^{p/(p-2)}(\Omega)$, $p \in [2,d]$,
  \begin{equation}\label{eq:eigen-var-char}
    \lambda_k(\nu,\mu) = \inf_{V_{k+1}\subset C^\infty(\overbar{\Omega})} \sup_{\phi \in V_{k+1}\setminus\brc{0}} \frac{\int \abs{\d\phi}^2 \d \nu }{\int \phi^2 \d\mu},
  \end{equation}
  and
  \begin{equation}
    \overbar{\lambda}_{k,p}(\nu,\mu) = \lambda_{k}(\nu,\mu)\frac{\mu(\overbar{\Omega})}{\norm{\nu}_{L^{\frac{p}{p-2}}}}.
  \end{equation}
  The critical measures for $\overbar{\lambda}_{k,p}(\nu,\mu)$ correspond to $p$-harmonic maps into spheres. Moreover,
  the case $p = d$ corresponds to optimization within the conformal class $[g]$ (see \cite{Vinokurov:2026:p-harm-and-conf-class-opt}):
  \begin{equation}
    \overbar{\lambda}_{k,d}(\rho^{d-2},\rho^d) = \lambda_k(\rho^2 g) \Vol_{\rho^2 g}(\Omega)^{2/d}.
  \end{equation}
  Thus, it is tempting to extend Theorem~\ref{thm:qualitative-stability} to the case of $p$-harmonic maps.

  To apply Corollaries~\ref{cor:subdiff-calc} and~\ref{cor:subdiff-eigenval} effectively, one needs sufficiently many compact perturbations
  that keep the operator within the constraint set. In particular, in the proof of Theorem~\ref{thm:qualitative-stability}, it is essential
  that we have enough compact perturbations of measures with respect to the underlying Hilbert space
  $H^1(\Omega)$. Indeed, every $\mu' \in L^\infty(\Omega)$ gives such a perturbation. The main difficulty with
  $\overbar{\lambda}_{k,p}(\nu,\mu)$-optimization is that one must now work
  with compact perturbations of $\mu$ on the Hilbert space
  \begin{equation}
      H^1(\Omega,\mu,\nu) := \overline{\set*{(\phi,d\phi)}{\phi \in C^\infty(\overbar{\Omega})}}^{L^2(\Omega,\mu)\oplus L^2(\Omega,\nu,T^*\Omega)}.
  \end{equation}
  Although $H^1(\Omega,\mu,1) = H^1(\Omega)$ by Proposition~\ref{prop:meas-bilinear} and Theorem~\ref{thm:subdiff-eigenval-all} if $\lambda_k(\mu) > 0$,
  the case of $p$-harmonic maps
  would require considering all the densities $\nu = \abs{\d u}^{p-2} \in L_+^{p/(p-2)}(\Omega)$ induced by $p$-harmonic maps into spheres.
  For $d\geq 3$, it remains open whether $\abs{\d u}$ can vanish on a set of positive Lebesgue measure, in which case
  we do not know whether the space $H^1(\Omega,\mu,\nu)$ admits sufficiently many compact perturbations.
  Another problem is that the variations of $\nu$ (that is, the variations of the Dirichlet energy) are never compact, so one cannot
  use Corollary~\ref{cor:subdiff-eigenval} and must instead work directly with Theorem~\ref{thm:subdiff-eigenval-all} and Corollary~\ref{cor:subdiff-composed-eigenval}.

\section{Preliminaries}
The abstract operator results are stated over complex Hilbert spaces. In the Laplace and Steklov
applications we apply their real versions to the real Hilbert spaces; these versions
follow by complexification.

\subsection{Projective tensor product}\label{sec:proj-tens-prod}
Let $E$ and $F$ be normed spaces.
Let $\t \in E \otimes F$ be an element of the algebraic tensor product of $E$ and $F$. Its projective norm is defined as
\begin{equation}
  \norm{\t}_\pi = \inf \set*{\sum \norm{x_i}\norm{y_i}}{\t = \sum x_i \otimes y_i}.
\end{equation}
The completion of $E \otimes_\pi F := (E \otimes F, \norm{\cdot}_{\pi})$ is denoted by $E \widehat{\otimes}_\pi F$ and is called the \emph{projective tensor product} of $E$ and $F$.
Note that $E \widehat{\otimes}_\pi F = \widehat{E} \widehat{\otimes}_\pi \widehat{F}$, where $\widehat{E}$ and $\widehat{F}$ are the corresponding
completions.
The projective tensor norm satisfies the universal property of the tensor product in the category of normed spaces:
a bilinear map $E\times F \to G$ is continuous if and only if
its linearization $E \otimes_\pi F \to  G$ is continuous and has the same norm.
\begin{equation}
  \begin{tikzcd}
    E \times F & E \otimes_\pi F \\
    & G
    \arrow[from=1-1, to=1-2]
    \arrow[from=1-1, to=2-2]
    \arrow[dashed, from=1-2, to=2-2]
  \end{tikzcd}
\end{equation}
Taking $G = \mathbb{K} \in \brc{\R,\C}$ gives the isometric identification
\begin{equation}
  (E \widehat{\otimes}_\pi F)^* = \mathfrak{Bil}[E,F].
\end{equation}
The algebraic tensor product $E^* \otimes F$ is canonically identified with the space of \emph{finite-rank operators} $\mathfrak{F}[E,F] \subset \mathfrak{L}[E,F]$, where
$ y^* \otimes x \mapsto (v \mapsto \brt{v,y^*} x)$. This embedding extends continuously to a map $E^* \widehat{\otimes}_\pi F \to \mathfrak{L}[E,F]$;
this extension is not injective in general. The image of this map, equipped with the corresponding quotient norm, is denoted by $\mathfrak{N}[E,F]$,
and the operators in $\mathfrak{N}[E,F]$ are called \emph{nuclear operators}.

Recall that for a complex vector space $E$, $\overbar{E}$ denotes its complex conjugate; that is, scalar multiplication is defined by
$\alpha \cdot x := \overline{\alpha} x$ for $x \in \overbar{E}$. There is a canonical antilinear isometry $E \mapsto \overbar{E}$.
If $E$ is the complexification of a real vector space, $E = E_1 \otimes_{\R} \C$, then the natural conjugation involution
$x \mapsto \overbar{x}$ gives a canonical linear isometry $E \mapsto \overbar{E}$.

Let $H$ be a complex Hilbert space. We identify $H^*$ with $\overbar{H}$ via the Riesz representation theorem; namely,
$x$ corresponds to the functional $y \mapsto \brt{y,x}_H$. Under this identification, the inner product
becomes a bilinear form $\brt{\cdot,\cdot}\colon H \times \overbar{H}\to \C$, and we have the following isometries:
\begin{itemize}
  \item $\mathfrak{Bil}[\overbar{H}, H] \cong \mathfrak{L}[H] \cong \mathfrak{Bil}[H, \overbar{H}]$;
  \item $H\widehat{\otimes}_{\pi}\overbar{H} \cong \mathfrak{N}[H]$.
\end{itemize}

The space of compact operators $\mathfrak{K}[H]$ is a predual of $\mathfrak{N}[H]$
(that is, $\mathfrak{K}[H]^* \approx \mathfrak{N}[H]$) with duality pairing given by
\begin{equation}\label{eq:duality}
  (\mathfrak{q} , \t) \mapsto \Tr \mathfrak{q}\t, \quad \mathfrak{q} \in \mathfrak{K}[H],\ \t \in \mathfrak{N}[H],
\end{equation}
while $\mathfrak{L}[H] = \mathfrak{N}[H]^*$.
Moreover, for $\t \in \mathfrak{N}[H]$, the projective norm coincides with the trace norm:
\begin{equation}
  \norm{\t}_\pi = \norm{\t}_{\mathfrak{N}} = \Tr \sqrt{\t^* \t}.
\end{equation}
In particular, for a positive operator represented by a tensor $\t = \sum x_i \otimes x_i$,
the norm is simply the trace:
\begin{equation}\label{eq:positive-tens-norm}
  \norm{\t}_{\pi} = \Tr \t = \sum \norm{x_i}^2.
\end{equation}
For a more detailed exposition of the projective tensor product, see, for example, \cite{Defant-Floret:1993:tensor-norms}.

\subsection{Clarke subdifferential}\label{subsec:clarke-subdif}
Let $E$ be a Banach space, $\overline{\R}:=\R\cup\brc{+\infty}$,
and let $f\colon E \to \overline{\R}$ be Lipschitz on a neighborhood of $x \in E$. Whenever
$f$ is originally defined only on a subset of $E$, we extend it by setting $f(y) = +\infty$ outside its domain.

The \emph{Clarke directional derivative} of $f$ at $x$ in the direction $v \in E$ is defined by
\begin{equation}
 f^{\circ}(x; v) = \limsup_{\tilde{x} \to x, \tau\downarrow 0}
 \frac{1}{\tau} \brs{f(\tilde{x} + \tau v) - f(\tilde{x})}.
\end{equation}
Since $v \mapsto f^{\circ}(x; v)$ is sublinear and Lipschitz,
it is the support function of a nonempty, convex, weak-$*$ compact subset of $E^*$.
The Clarke subdifferential $\partial_c f(x)$ is defined to be this subset. Namely,
\begin{equation}
  \partial_c f(x) = \set*{\xi \in E^*}{\brt{\xi, v}
  \leq f^{\circ}(x; v) \ \forall v \in E}.
\end{equation}
The following properties of the Clarke subdifferential can be found, for example, in \cite[Sections~7.3,~7.4]{Schirotzek:2007:nonsmooth-analysis}:
\begin{itemize}
  \item $\partial_c f(x) \neq \varnothing$ and $\partial_c f(x) \subset (\mathrm{Lip}_x f) B_{E^*}$, where $B_{E^*} \subset E^*$ is the unit ball;
  \item $\partial_c f(x)$ is a convex and weak-$*$ compact set;
  \item $\partial_c (\alpha f)(x) = \alpha\partial_c f(x)$ for $\alpha \in \R$;
  \item $\partial_c (f+g)(x) \subset \partial_c f(x) + \partial_c g(x)$;
  \item $\partial_c (fg)(x) \subset g(x)\partial_c f(x) + f(x)\partial_c g(x)$ (follows from \cite[Theorem~7.4.5]{Schirotzek:2007:nonsmooth-analysis});
  \item $\partial_c f(x) = \brc{\d_x f}$ if $f \in C^1$ on a neighborhood of $x$;
  \item $0 \in \partial_c f(x)$ if $x$ is a local maximum or minimum.
\end{itemize}
The Clarke subdifferential extends the notion of derivative to locally Lipschitz functions
and retains many properties of derivatives that are useful in optimization.
For applications of the Clarke subdifferential to critical metrics and eigenvalue optimization,
see~\cite{Petrides-Tewodrose:2024:eigenvalue-via-clarke,Vinokurov:2025:higher-dim-harm-eigenval,Vinokurov:2026:p-harm-and-conf-class-opt}.

The following chain rule can be found, for example, in \cite[Proposition~2.2]{Vinokurov:2025:sym-eigen-val-lms}.
\begin{proposition}[Chain rule]\label{prop:clarke-compose}
  Let $G \colon F \to E$ be a continuously Fréchet differentiable map between Banach spaces defined
  on a neighborhood of $x\in F$. If
  $f\colon E \to \overline{\R}$ is Lipschitz
  on a neighborhood of $y = G(x) \in E$, one has
  \begin{equation}
    \partial_c (f \circ G)(x) \subset \d_x G^*[\partial_c f (y)].
  \end{equation}
  In other words, the Clarke subdifferential commutes with pullbacks up to inclusion.
\end{proposition}

\begin{proposition}[{\cite[Propositions~4.6.2 and~7.3.9(d)]{Schirotzek:2007:nonsmooth-analysis}}]
  \label{prop:subdiff-of-norm}
  Let $E$ be a Banach space, and $\omega(x) := \norm{x}$. Then
  \begin{equation}
    \partial_c\omega(x) = \begin{dcases}
      \set*{x^* \in E^*}{\norm{x^*} = 1,\ \brt{x,x^*} = \norm{x}} &\quad \text{if } x \neq 0,    \\
      B_{E^*}                                                     &\quad x=0,
    \end{dcases}
  \end{equation}
  where $B_{E^*}$ is the unit ball of $E^*$.
\end{proposition}

The following proposition is a simplified version of~\cite[Theorem~12.4.1]{Schirotzek:2007:nonsmooth-analysis}
sufficient for our purposes (see also \cite{Clarke:1976:lagrange-multipliers}).
\begin{proposition}[Clarke’s Multiplier Rule]\label{prop:multip-rule}
  Let $E$ be a Banach space, and let $f \colon E \to \overline{\R}$ be a function
  that is Lipschitz on a neighborhood of $\overbar{x} \in S$, where $S \subset E$ is a closed
  convex set. If $\overbar{x}$ is a local minimum of $f$ on $S$, then
  \begin{equation}
    \exists x^* \in \partial_c f(\overbar{x}) \colon \
    \brt{x-\overbar{x}, x^*} \geq 0 \quad\forall x \in S.
  \end{equation}
\end{proposition}

Note that for unbounded self-adjoint operators, eigenvalues are indexed in
increasing order (from bottom to top). In contrast, for
bounded operators, we will index the eigenvalues in decreasing order
(from top to bottom), starting with $k=1$. The proof of the next proposition follows the arguments of
\cite[Proposition~2.18]{Vinokurov:2025:higher-dim-harm-eigenval} and \cite[Lemma~2.5]{Vinokurov:2025:sym-eigen-val-lms} and uses only
that $\lambda_k(T)$ is an isolated eigenvalue of finite multiplicity.
\begin{proposition}
  \label{prop:subdiff-calc}
  If $S \in \mathfrak{L}_{sa}[H]$ is a bounded operator with an isolated eigenvalue $\lambda_k(S)$ of finite multiplicity, then the
  Clarke subdifferential of $\lambda_k\colon \mathfrak{L}_{sa}[H] \to \R$ at $S$, $\partial_c \lambda_k(S) \subset \mathfrak{N}_{sa}[H] \subset \mathfrak{L}_{sa}[H]^*$,
  is given by
  \begin{equation}\label{abst-clark-subdif}
    \partial_c \lambda_k(S) = \operatorname{co}\,
    \set*{x \otimes x }{
      x \in E_{\lambda_k}(S),\ \norm{x} = 1
    },
  \end{equation}
  where $\operatorname{co} K$ denotes the convex hull of the set $K$.
\end{proposition}

\subsection{Dual functionals on \texorpdfstring{$\mathfrak{L}[H]$}{Lg}}\label{sec:operator-dual}
In this section, we use $T^* \in \mathfrak{L}[H]$ to denote the Hilbert-space adjoint of an operator $T \in \mathfrak{L}[H]$.

We have $\mathfrak{L}_{sa}[H] = \set{T \in \mathfrak{L}[H]}{T^* = T}$, and $\mathfrak{L}[H] = \mathfrak{L}_{sa}[H] \otimes_{\R} \C$. Indeed,
under the projection $\mathrm{Re} \colon \mathfrak{L}[H] \to \mathfrak{L}_{sa}[H]$, $T \mapsto \frac{1}{2}(T + T^*)$,
every operator $T \in \mathfrak{L}[H]$ admits a decomposition $T = \mathrm{Re} T + i \mathrm{Re} (i^{-1}T)$; moreover,
$\norm{T}$ is equivalent to $\sup_{\abs{\lambda}=1}\norm{\mathrm{Re} (\lambda T)}$. One readily checks that
\begin{equation}\label{eq:self-adj-func}
  \mathfrak{L}_{sa}[H]^* = \set*{\phi \in \mathfrak{L}[H]^*}{\phi^* = \phi},
\end{equation}
where $\phi^*(T) := \overbar{\phi(T^*)}$ is the involution on $\mathfrak{L}[H]^*$ induced by
the involution on $\mathfrak{L}[H]$.

Consequently, every $\psi \in \mathfrak{L}[H]^*$ can be written as $\psi = \mathrm{Re} \psi + i \mathrm{Re} (i^{-1}\psi)$, and
$\mathrm{Re} \psi \in \mathfrak{L}_{sa}[H]^*$ is a self-adjoint functional (hence real-valued on $\mathfrak{L}_{sa}[H]$).
Furthermore, every $\phi \in \mathfrak{L}_{sa}[H]^*$ can be decomposed uniquely as the difference of two positive functionals
(a functional $\phi$ is positive if $T \geq 0 \implies \phi(T) \geq 0$):
\begin{equation}
  \phi = \phi_+ - \phi_-
  \quad\text{such that}\quad
  \norm{\phi} = \norm{\phi_+} + \norm{\phi_-};
\end{equation}
see \cite[Theorem~4.3.6 and Remark~4.3.12]{Kadison-Ringrose:1997:operator-algebras-vol-I}.
For a self-adjoint functional $\phi \in \mathfrak{L}_{sa}[H]^*$,
we denote by $\abs{\phi} \in \mathfrak{L}_{sa}[H]^*$ the functional $\abs{\phi} := \phi_+ + \phi_-$. Hence $\phi \leq \abs{\phi}$.

Recall that every Banach space $E$ embeds isometrically into its bidual $\iota_E\colon E \hookrightarrow E^{**}$. By duality, there is also
the canonical projection $(\iota_{E^*}\circ\iota_E^*) \colon E^{***} \twoheadrightarrow  E^* \hookrightarrow E^{***}$. It then follows
that $E^{***} \approx E^* \oplus E^{\perp}$, where $E^{\perp} \subset E^{***}$ is the subspace of functionals vanishing on $E$.
\begin{remark}\label{rem:m-ideals}
  If $E^{***} = E^* \oplus_1 E^{\perp}$, the subspace $E \subset E^{**}$ is called an \emph{M-ideal}. This holds whenever
  $E = \mathfrak{K}[H]$ is the space of compact operators on a Hilbert space $H$ (with $E^{**} = \mathfrak{L}[H]$);
  see \cite[Example~I.1.4(d)]{Harmand-Werner-Werner:1993:m-ideals} or \cite[Theorem~III.2.14]{Takesaki:2002:operator-algebras}. The same decomposition holds for self-adjoint operators
  by~\eqref{eq:self-adj-func}.
\end{remark}

A positive functional
$\phi$ such that $\norm{\phi} = 1$ is called a \emph{state}. Note that for a positive functional, $\phi(1) = \norm{\phi}$, since
$1 \geq T$ for every self-adjoint $T$ with $\norm{T} \leq 1$. Observe that for every state $\phi$,
the sesquilinear form $(A,B) \mapsto \phi(B^*A)$, where $A,B \in \mathfrak{L}[H]$, is Hermitian and nonnegative definite.
Consequently, the Cauchy--Schwarz inequality yields
\begin{equation}\label{ineq:state-cauchy}
  \abs{\phi(A^*B)}^2 \leq \phi(A^*A) \phi(B^*B).
\end{equation}

Let $H_1\subset H$ be a closed linear subspace of $H$. We say that a functional $\psi \in \mathfrak{L}[H]^*$ is \emph{supported on } $H_1$
and write $\supp \psi \subset H_1$ if
\begin{equation}\label{eq:state-support}
  \psi(T) = \psi(P_{H_1}TP_{H_1}) \quad \forall T \in \mathfrak{L}[H],
\end{equation}
where $P_{H_1}$ is the orthogonal projection onto $H_1$. Equivalently, $\psi$ may be regarded as an element of $\mathfrak{L}[H_1]^*$,
extended to $\mathfrak{L}[H]$ by~\eqref{eq:state-support}.
\begin{proposition}\label{prop:max-func}
  For a positive operator $S \in \mathfrak{L}_+[H]\setminus\brc{0}$, consider the set of maximizing functionals
  \begin{equation}
    M_S := \set*{\phi \in \mathfrak{L}_{sa}[H]^*}{\norm{\phi} = 1,\ \phi(S) = \norm{S}}.
  \end{equation}
  Let $E_{(\lambda_1 - \epsilon, \lambda_1]}$ be the image of the spectral projection $\chi_{(\lambda_1 - \epsilon, \lambda_1]}(S)$, where
  $\lambda_1 = \norm{S}$. Then
  \begin{equation}\label{eq:max-states-desc}
    M_S = \set*{\phi \in \mathfrak{L}_{sa}[H]^*}{\norm{\phi} = 1,\ \phi \geq 0,\
    \supp \phi \subset E_{(\lambda_1 - \epsilon, \lambda_1]}\ \forall \epsilon > 0}.
  \end{equation}
\end{proposition}
\begin{proof}
  If $\phi \in M_S$, then $\norm{S} =\phi(S) \leq \abs{\phi}(S) \leq \norm{\phi} \norm{S} = \norm{S}$. Therefore, $\phi_-(S) = 0$.
  We then have that $\norm{S} =\phi_+(S) \leq \norm{\phi_+} \norm{S} \leq \norm{S}$. Thus, $\norm{\phi_+} = 1$, which yields $\phi_- = 0$.
  Thus $M_S$ lies in the set of states.

  Let $P = \chi_{(\lambda_1 - \epsilon, \lambda_1]}(S)$ and $Q = 1 - P$; then $\lambda_1 - S \geq \epsilon Q$. Hence,
  $0\leq \epsilon \phi(Q) \leq \lambda_1\phi(1) - \phi(S)= 0$, so $\phi(Q) = 0$. Given $T \in \mathfrak{L}[H]$, write
  $T = PTP + QTP + PTQ + QTQ$ and use the Cauchy--Schwarz inequality~\eqref{ineq:state-cauchy} to conclude that $\phi$ vanishes on the last three terms.
  This proves the inclusion $M_S \subset \cdots$.

  To see that the opposite inclusion holds, take a state $\phi$ belonging to the right-hand side of \eqref{eq:max-states-desc}. Then, for every $\epsilon > 0$, we have
  $\lambda_1 - \epsilon \leq PS P \leq \lambda_1$, so $\phi(S) = \phi(PSP) \in [\lambda_1 - \epsilon, \lambda_1]$.
\end{proof}
By \cite[Theorem~4.3.9(ii-iii)]{Kadison-Ringrose:1997:operator-algebras-vol-I},
the weak-$*$ continuous states (that is, the positive nuclear operators of trace $1$) are weak-$*$ dense in the set of all states. In other words,
\begin{equation}
  \set*{\phi \in \mathfrak{L}_{sa}[H]^*}{\norm{\phi} = 1,\ \phi \geq 0} = \overline{\operatorname{co}}^{w^*}\set*{x\otimes x}{x \in H,\ \norm{x} = 1}.
\end{equation}
Applying this observation to $H = E_{(\lambda_1 - \epsilon, \lambda_1]}$, we obtain
\begin{corollary}\label{cor:max-func}
  In the context of Proposition~\ref{prop:max-func},
  \begin{equation}
    M_S = \bigcap_{\epsilon > 0} \overline{\operatorname{co}}^{w^*}\set*{x\otimes x}{x \in E_{(\lambda_1 - \epsilon, \lambda_1]},\ \norm{x} = 1}.
  \end{equation}
\end{corollary}

\subsection{Measures as bilinear forms}\label{sec:meas-as-forms}
\begin{lemma}\label{lem:unstable-ring}
  Let $\Omega \subset M$ be a domain in a Riemannian manifold $(M,g)$, and let $0\neq\mu \in \mathcal{M}_+^{c}(\overbar{\Omega})$ with $\lambda_k(\Omega, \mu) =1$ for some $k > 0$.
  Then, for every point $p \in \overbar{\Omega}$, there exists a neighborhood $U$ such that for all
    $\phi \in \D(\overbar{\Omega})$ with $\supp \phi \subset U \cap\overbar{\Omega}$,
    one has
    \begin{equation}\label{eq:unstable-ring}
      \int_{\overbar{\Omega}} \phi^2 \mathrm{d}\mu \leq \int_{\Omega} \abs{\mathrm{d}\phi}^2.
    \end{equation}
\end{lemma}
\begin{proof}
  Note that it suffices to prove \eqref{eq:unstable-ring} for $\supp \phi \subset (U\setminus\brc{p}) \cap\overbar{\Omega}$, since
  discrete sets have zero capacity and $\mu(\brc{p}) = 0$.

  Arguing by contradiction, we can find a sequence of functions $\phi_i \in \D(\overbar{\Omega})$ with disjoint supports
  $\supp \phi_i \subset U_i \cap\overbar{\Omega}$, where $U_i := B_{r_i}(p)\setminus\overline{B_{r_{i+1}}(p)}$ and $r_i \searrow 0$, such that
  \begin{equation}
    \int_{\Omega} \abs{\mathrm{d}\phi_i}^2 - \int_{\overbar{\Omega}} \phi_i^2 \mathrm{d}\mu < 0.
  \end{equation}
  This contradicts the variational characterization of $\lambda_k(\Omega, \mu)=1$ once $k+1$ such functions have been constructed.
  Therefore, \eqref{eq:unstable-ring} holds if we take $U := B_{r_{k+1}}(p)$.
\end{proof}
\begin{proposition}\label{prop:meas-bilinear}
    Let $\Omega \subset M$ be a bounded domain in a Riemannian manifold $(M,g)$, and let $0\neq\mu \in \mathcal{M}_+^{c}(\overbar{\Omega})$.
    If $\D(\overbar{\Omega})$ is dense in $H^1(\Omega)$ and $\lambda_k(\Omega, \mu) \neq 0$ for some $k > 0$, then the measure
    $\mu$ induces a continuous bilinear form on $H^1(\Omega)$, that is, $\mu \in \mathfrak{Bil}[H^1(\Omega)]$.
\end{proposition}
\begin{proof}
    By Lemma~\ref{lem:unstable-ring}, every point $p \in \overbar{\Omega}$ has a neighborhood $U$ such that for all
    $\phi \in \D(\overbar{\Omega})$ with $\supp \phi \subset U \cap\overbar{\Omega}$, one has
    \begin{equation}
        \lambda_k(\Omega, \mu)\int_{\overbar{\Omega}} \phi^2 \mathrm{d}\mu \leq \int_{\Omega} \abs{\mathrm{d}\phi}^2.
    \end{equation}
    A partition-of-unity argument with $\sum_i \eta_i^2 = 1$,
    $\supp \eta_i \subset U_i$, and $\overbar{\Omega} \subset \bigcup_i U_i$, shows that for all $\phi \in \D(\overbar{\Omega})$,
    \begin{equation}
        \lambda_k(\Omega, \mu)\int_{\overbar{\Omega}} \phi^2 \mathrm{d}\mu \leq \int_{\Omega} \abs{\mathrm{d}\phi}^2
        + \frac{1}{2} \sum_i \int_{\Omega} \brt{\mathrm{d}\eta_i^2, \mathrm{d}\phi^2} + \sum_i \int_{\Omega} \phi^2\abs{\mathrm{d}\eta_i}^2,
    \end{equation}
    where the middle sum vanishes. Hence, there exists a constant $C > 0$ such that
    \begin{equation}
        \lambda_k(\Omega, \mu)\int_{\overbar{\Omega}} \phi^2 \mathrm{d}\mu \leq \int_{\Omega} \abs{\mathrm{d}\phi}^2 + C\int_{\Omega} \phi^2.
    \end{equation}
\end{proof}
Therefore, the canonical map $\D(\overbar{\Omega})\to L^2(\overbar{\Omega},\mu)$ extends uniquely to a continuous linear map $H^1(\Omega) \to L^2(\overbar{\Omega},\mu)$.
Integration with respect to $\mu$ will be understood via this map.

\section{Proofs of the main results}

\subsection{Abstract setting}
\subsubsection{Theorem~\ref{thm:subdiff-calc-all}}
  By replacing $S$ with $S + c\cdot 1$ for sufficiently large $c >0$, we may assume that $S \geq 0$ and $\lambda_k(S) > 0$.

  We begin with the case $k=1$. Then $\lambda_1(S') = \norm{S'}$ for any $S' \in \mathfrak{L}_{sa}[H]$ in a neighborhood of $S$.
  Combining Proposition~\ref{prop:subdiff-of-norm} with Corollary~\ref{cor:max-func}, we prove the theorem for $\lambda_1$. Namely,
  \begin{equation}
    \begin{aligned}
      \partial_c \lambda_1(S) &= \set*{\phi \in \mathfrak{L}_{sa}[H]^*}{\norm{\phi} = 1,\ \phi \geq 0,\ \supp \phi \subset E_{(\lambda_1 - \epsilon, \lambda_1]}(S)\ \forall \epsilon > 0}
      \\                      &= \bigcap_{\epsilon > 0} \overline{\operatorname{co}}^{w^*}\set*{x\otimes x}{x \in E_{(\lambda_1 - \epsilon, \lambda_1]}(S),\ \norm{x} = 1}.
    \end{aligned}
  \end{equation}

  For $k > 1$, denote by $P_S$ the orthogonal projector onto the (finite-dimensional) space $\bigoplus_{\lambda > \lambda_k}E_\lambda$,
  where $E_\lambda = E_\lambda(S)$.
  Set $Q_S := 1 - P_S$.
  Since $P_S$ is given by the contour integral of the resolvent over a curve enclosing all eigenvalues $\lambda > \lambda_k$,
  we see that $P_S$ depends analytically on $S$. In particular,
  the map $\gamma \colon S' \mapsto Q_{S'} S' Q_{S'}$
  is of class $C^1$ on a neighborhood of $S \in \mathfrak{L}_{sa}[H]$.
  Note that the differential of $\gamma$ at $S$ equals
  \begin{equation}\label{eq:proj-diff}
    (\d_S\gamma)(T) = Q_STQ_S - (\d_S P)(T)SQ_S - Q_S S (\d_S P)(T),
  \end{equation}
  and $(\d_S P)(T)$ satisfies $Q_S (\d_S P)(T) Q_S = 0$, as follows from differentiating the identity
  $Q_S^2 = Q_S$ and multiplying by $Q_S$ on both sides.

  Then the condition $0< \lambda_k(S) < \lambda_{k-1}(S)$ implies that
  \begin{equation}
    \lambda_k(S') = \lambda_1(Q_{S'}S'Q_{S'}) = \norm{Q_{S'}S'Q_{S'}}
  \end{equation}
  for $S'$ in a neighborhood of $S$.
  The inclusion $\partial_c \lambda_k(S) \subset \cdots$ now follows from the case $k = 1$,
  Proposition~\ref{prop:clarke-compose}, and \eqref{eq:proj-diff}, since
  $E_{(\lambda_1 - \epsilon, \lambda_1]}(Q_S S Q_S) = E_{(\lambda_k - \epsilon, \lambda_k]}(S)$ and
  $\d_S\gamma^*(\phi) = \phi$ for $\phi$ supported on $im\, Q_S = E_{(-\infty, \lambda_k]}(S)$:
  \begin{equation}\label{eq:support-calc-compact}
    \brt{T, \d_S\gamma^*(\phi)} = \phi\brr{\d_S\gamma(T)} = \phi\brr{Q_S(\d_S\gamma)(T)Q_S} = \phi\brr{Q_STQ_S} = \phi(T).
  \end{equation}
  To prove the reverse inclusion $\partial_c \lambda_k(S) \supset \cdots$, let $H_1 := E_{(-\infty, \lambda_k]}(S)$ and
  consider the affine map $\iota\colon \mathfrak{L}_{sa}[H_1] \to \mathfrak{L}_{sa}[H]$, $T \mapsto Q_STQ_S + P_S S P_S$. Then $\lambda_1(T) = \lambda_k(\iota(T))$
  for $T$ in a neighborhood of $Q_S S Q_S \in \mathfrak{L}_{sa}[H_1]$, and the chain rule (Proposition~\ref{prop:clarke-compose}) implies
  $\partial_c\lambda_1(Q_S S Q_S) \subset (\d_{Q_S S Q_S} \iota)^*[\partial_c \lambda_k(S)] = \partial_c \lambda_k(S)$, since
  all the states $\phi \in \partial_c \lambda_k(S)$ are supported on $H_1$.

  \subsubsection{Theorem~\ref{thm:subdiff-eigenval-all}}
  Let us prove that $\mathfrak{a} + c \mathfrak{b}$ defines an equivalent inner product on $H$ for some sufficiently large $c>0$.
  Define an equivalent norm $\norm{x}^2 := \mathfrak{a}[x] + \mathfrak{r}[x]$, where $\mathfrak{r}$ is as in Remark~\ref{rem:positive-ess-spec}.
  If there is no such constant $c > 0$, we can find a sequence $\norm{x_n} = 1$ such that
  \begin{equation}\label{ineq:contr-lower-bound}
    \mathfrak{a}[x_n] + n\mathfrak{b}[x_n] < \frac{1}{n}.
  \end{equation}
  Passing to a subsequence if necessary, we may assume that $x_n \oset{w}{\to} x$ with $\norm{x} \leq 1$.
  Since $\mathfrak{r}$ is compact, one sees that
  \begin{equation}\label{ineq:contr-a-x_n}
    1/n > \mathfrak{a}[x_n] = 1 - \mathfrak{r}[x_n] \to 1 - \mathfrak{r}[x] = \mathfrak{a}[x] + (1 - \norm{x}^2).
  \end{equation}
  Hence $\mathfrak{a}[x] \leq 0$.
  Then~\eqref{ineq:contr-lower-bound} also implies
  \begin{equation}
    n\mathfrak{b}[x_n] \leq \norm{x_n}^2 + n\mathfrak{b}[x_n] \leq \frac{1}{n} + \mathfrak{r}[x_n] \leq C.
  \end{equation}
  Hence $0 \leq \mathfrak{b}[x] \leq \liminf_n \mathfrak{b}[x_n] = 0$. By assumption,
  $x \in \ker \mathfrak{b} \cap \set{x\in H}{\mathfrak{a}[x] \leq 0}$ implies $x  = 0$, which contradicts \eqref{ineq:contr-a-x_n}.

  Therefore, there exists $c > 0$ such that $\mathfrak{a} + c \mathfrak{b} \geq \frac{1}{c} \norm{\cdot}^2$.
  We now equip $H$ with the inner product $\mathfrak{a} + c \mathfrak{b}$, and for $(\mathfrak{a}', \mathfrak{b}')$ in a neighborhood of $(\mathfrak{a}, \mathfrak{b}) \in \brr{Herm\,[H]}^{\times 2}$,
  we obtain
  \begin{equation}
    \frac{1}{\lambda_k(\mathfrak{a}',\mathfrak{b}') + c}
    = \sup_{V_{k} \subset D} \inf_{x \in V_{k}\setminus\brc{0}} \frac{\mathfrak{b}'[x]}{(\mathfrak{a}'
    + c \mathfrak{b}')[x]} = \lambda_k\brr{(A'+cB')^{-1/2}B'(A'+cB')^{-1/2}},
  \end{equation}
  where $A',B'$ are the self-adjoint operators associated with the Hermitian forms $\mathfrak{a}',\mathfrak{b}'$ with respect to the inner product $\mathfrak{a} + c \mathfrak{b}$.
  In particular, $B = (\mathfrak{a} + c\mathfrak{b})^{-1}\mathfrak{b}$ and $A + c B = 1$. Moreover, $S := (A'-A)+c(B'-B)$ lies in a neighborhood of $0 \in \mathfrak{L}_{sa}[H]$, so
  $(A'+cB')^{-1/2} = (1 + S)^{-1/2} = 1 - \frac{1}{2} S + \frac{3}{8}S^2 + \cdots$.

  The map $f \colon (\mathfrak{a}',\mathfrak{b}') \mapsto (A'+cB')^{-1/2}B'(A'+cB')^{-1/2}$ is analytic, and we have
  \begin{equation}
    \d_{(\mathfrak{a},\mathfrak{b})} f (\mathfrak{h}_a,\mathfrak{h}_b) = H_b -\frac{1}{2}\brr{(H_a+cH_b) B + B (H_a+ cH_b)},
  \end{equation}
  where $\mathfrak{h}_a$ and $H_a$ are related via the inner product $\mathfrak{a} + c \mathfrak{b}$.
  Since $\lambda_{k-1}(\mathfrak{a},\mathfrak{b}) < \lambda_{k}(\mathfrak{a},\mathfrak{b}) < \infty$, Corollary~\ref{cor:subdiff-composed-eigenval} yields
  \begin{equation}
    \begin{multlined}
      \frac{1}{(\lambda_k + c)^2}
      \partial_c (-\lambda_k)\brr{\mathfrak{a}, \mathfrak{b}}
      \\ \subset \bigcap_{\epsilon > 0} \overline{\operatorname{co}}^{w^*}
        \set*{\brr{- \frac{x\otimes x}{\lambda_k + c}, \frac{\lambda_k x\otimes x}{\lambda_k + c} }}{x\in E_{[\lambda_k, \lambda_k + \epsilon)}(\mathfrak{a}, \mathfrak{b}),\ \mathfrak{b}[x] = \frac{1}{\lambda_k + c}}.
    \end{multlined}
  \end{equation}
  It remains to multiply both sides by $(\lambda_k + c)^2$.
\begin{remark}
  In general, if we define the variational eigenvalues  $\lambda_k(\mathfrak{a}, \mathfrak{b})$ as above and
  $\mathfrak{a}$ is an inner product, then there are two natural ways to construct an operator whose spectrum contains
  $\lambda_k^{-1}(\mathfrak{a}, \mathfrak{b})$. This can be done either by setting
  \begin{equation}
    T_1 = \mathfrak{a}^{-1}\mathfrak{b} = A^{-1}B,
    \quad\text{or}\quad
    T_2 = A^{-1/2}BA^{-1/2},
  \end{equation}
  where $A$ and $B$ are the corresponding self-adjoint operators under the identification $H^* \approx \overbar{H}$. The operator $T_1$ appears to be the more canonical choice.
  However, $T_2$ has the advantage of being self-adjoint. Indeed, $T_1 = A^{-1/2}T_2A^{1/2}$; hence $\sigma(T_1) = \sigma(T_2)$.
\end{remark}

\subsubsection{Corollaries~\ref{cor:subdiff-calc} and~\ref{cor:subdiff-eigenval}}
  Let $\iota \colon \mathfrak{K}_{sa}[H] \hookrightarrow \mathfrak{L}_{sa}[H]$ be the canonical embedding. Then
  the adjoint $\iota^* \colon \mathfrak{L}_{sa}[H]^* \approx \mathfrak{N}_{sa}[H]\oplus \mathfrak{N}_{sa}[H]^{\perp} \to \mathfrak{N}_{sa}[H]$
  acts as the restriction map on $\mathfrak{N}_{sa}[H]$. Write $E_{(\lambda_k - \epsilon, \lambda_k]} = E_{(\lambda_k - \epsilon, \lambda_k]}(S)$, and use analogous notation for the remaining spectral subspaces. By applying Corollary~\ref{cor:subdiff-composed-eigenval} to $K \mapsto \iota(K) + S$, we obtain
  \begin{equation}
    \begin{aligned}
      \partial_c\lambda_k(S + 0) &\subset \bigcap_{\epsilon > 0} \overline{\operatorname{co}}^{w^*}\set*{x\otimes x}{x \in E_{(\lambda_k - \epsilon, \lambda_k]},\ \norm{x} = 1}
      \\                         &\subset \bigcap_{\epsilon > 0} \set*{T \in \mathfrak{N}_{+}[E_{(\lambda_k - \epsilon, \lambda_k]}]}{\norm{T}_{\mathfrak{N}} \leq 1}
      \\                         & = \set*{T \in \mathfrak{N}_{+}[E_{\lambda_k}]}{\norm{T}_{\mathfrak{N}} \leq 1},
    \end{aligned}
  \end{equation}
  where the weak-$*$ closures are taken in $\mathfrak{N}_{sa}[H]$ and $(\d_0 \iota)^* (x\otimes x) = x\otimes x$.
  The last equality follows from the fact that $T \in \bigcap_{\epsilon > 0} \mathfrak{N}[E_{(\lambda_k - \epsilon, \lambda_k]}] \implies
  im\, T \subset \bigcap_{\epsilon > 0} E_{(\lambda_k - \epsilon, \lambda_k]} = E_{\lambda_k}$. This completes the proof of Corollary~\ref{cor:subdiff-calc}.

  Corollary~\ref{cor:subdiff-eigenval} is proved similarly, since we also have
  \begin{equation}
    \t \in \bigcap_{\epsilon > 0}E_{(\lambda_k - \epsilon, \lambda_k]} \widehat{\otimes}_{\pi} \overbar{E}_{(\lambda_k - \epsilon, \lambda_k]} \implies
    im\, \t \subset \bigcap_{\epsilon > 0} E_{(\lambda_k - \epsilon, \lambda_k]} = E_{\lambda_k},
  \end{equation}
  where we identify $E_{(\lambda_k - \epsilon, \lambda_k]} \widehat{\otimes}_{\pi} \overbar{E}_{(\lambda_k - \epsilon, \lambda_k]}
  \approx \mathfrak{N}[E_{(\lambda_k - \epsilon, \lambda_k]}]$
  via the inner product on $H$. Hence, $\t$ admits a spectral decomposition $x_i\in E_{\lambda_k}$, $\t = \sum_i x_i \otimes x_i$.

\subsection{Eigenvalues of measures}
\subsubsection{Theorem~\ref{thm:qualitative-stability}}
  Without loss of generality, we may assume that $\overbar{\lambda}_{k-1}(\mu) < \overbar{\lambda}_k(\mu)$.
  Indeed, if $\overbar{\lambda}_{k-1}(\mu) = \overbar{\lambda}_k(\mu)$ and $\nu \in \mathcal{M}_+^c(\overbar{\Omega})$ is arbitrary, then
  $\overbar{\lambda}_{k-1}(\nu) \leq \overbar{\lambda}_{k}(\nu)
  \leq \overbar{\lambda}_{k}(\mu)  = \overbar{\lambda}_{k-1}(\mu)$, which implies $\mu$ is also a maximizer for $\overbar{\lambda}_{k-1}$, so it suffices to replace $k$ by $k-1$.

  If $\overbar{\lambda}_k(\mu) = \Lambda_k(\Omega,g) > 0$, then $\mu \in \mathfrak{Bil}[H^1(\Omega)]$ by Proposition~\ref{prop:meas-bilinear}. In particular,
  there is a bounded linear map $H^1(\Omega) \to L^2(\mu)$ extending $\id\colon \D(\overbar{\Omega}) \to L^2(\mu)$ and preserving the lattice structure.

  Let us rescale $\mu$ so that $\mu(\overbar{\Omega}) = 1$.
  Consider the function $L^\infty_+(\Omega) \ni \rho \mapsto \overbar{\lambda}_k(\mu + \rho) =
  \lambda_k(\mu + \rho)(1 + \int_{\Omega} \rho)$. If $E_k(\mu) \subset H^1(\Omega)$
  denotes the $k$th eigenspace of $\Delta \phi = \lambda \phi\mu$, then the Leibniz rule, together with Corollary~\ref{cor:subdiff-eigenval} and the chain rule, yields
  \begin{equation}\label{eq:laplace-subdiff}
    \begin{aligned}
      \partial_c(-\overbar{\lambda}_k)(\mu + 0) &\subset \lambda_k \cdot\set*{\sum_i \phi_i^2 - 1}{\phi_i \in E_k(\mu) ,\ \sum_i\int_{\overbar{\Omega}} \phi_i^2 \d\mu \leq 1}
      \\                                        & =  \lambda_k \cdot\set*{\abs{u}^2 - 1}{u \in H^1(\Omega,\ell^2),\ \Delta u = \lambda u\mu,\ \int_{\overbar{\Omega}} \abs{u}^2 \d\mu \leq 1}.
    \end{aligned}
  \end{equation}
  Applying Proposition~\ref{prop:multip-rule} with $S = L^\infty_+(\Omega)$ and $\overbar{x} = 0$, we conclude that there exists $u \in H^1(\Omega, \ell^2)$ such that $\Delta u = \lambda_k u\mu$,
  $\int_{\overbar{\Omega}} \abs{u}^2d\mu \leq 1$, and
  \begin{equation}
    \int_{\Omega} (\abs{u}^2-1)\rho \geq 0 \quad \forall \rho \in L^\infty_+(\Omega).
  \end{equation}
  Therefore, $\abs{u} \geq 1$ in $H^1(\Omega)$ and $L^2(\mu)$. Combining this with $\int_{\overbar{\Omega}} \abs{u}^2\d\mu \leq 1$,
  we see that $\abs{u} = 1$ in $L^2(\mu)$. Since $\d \abs{u} = \frac{u}{\abs{u}} \cdot \d u$ and
  \begin{equation}
    \brt{\d\abs{u}, \d\phi} =
  \brt{\d u, \frac{u}{\abs{u}} \d \phi} = \brt{\d u, \d\brr{\frac{u}{\abs{u}}\phi}} - \phi\brt{\d u,\d\brr{\frac{u}{\abs{u}}}},
  \end{equation}
  where a positive term is subtracted, we obtain the following form of the weak maximum principle:
  \begin{equation}\label{ineq:max-principle}
    \int_{\Omega}\brt{\d\abs{u}, \d\phi} \leq \lambda_k \int_{\overbar{\Omega}} \abs{u}\phi \d\mu \quad \forall \phi \in \D_+(\overbar{\Omega}).
  \end{equation}
  This inequality extends to every $\phi \in H^1_+(\Omega)$; testing it against $\abs{u} - 1 \geq 0$ gives
  \begin{equation}
    0\leq \int_{\Omega} \abs{\d\abs{u}}^2 \leq \lambda_k\int_{\overbar{\Omega}} (\abs{u} - 1)\abs{u}\d\mu = 0.
  \end{equation}
  Thus, $\abs{u} = 1$ in $H^1(\Omega)$. Now, we can test the equation $\Delta u = \lambda_k u \mu$ against $\phi u$, where
  $\phi \in \D(\overbar{\Omega})$. Since $u\cdot \d u = 0$, we obtain
  \begin{equation}
    \int_{\Omega} \phi\abs{\d u}^2 = \lambda_k\int_{\overbar{\Omega}} \phi \d\mu \quad \forall \phi \in \D(\overbar{\Omega}).
  \end{equation}
  Hence, $\lambda_k\mu = \abs{\d u}^2 \in L^1(\Omega)$ and $u$ is harmonic up to the boundary, that is, $\Delta u  = \abs{\d u}^2u$ in $\D(\overbar{\Omega})^*$.

  The interior regularity of $u$ then follows from~\eqref{eq:unstable-ring} and \cite[Lemma~5.5, Theorem~5.14, Corollary~1.3]{Vinokurov:2025:higher-dim-harm-eigenval}.
  When $\partial \Omega = \varnothing$, \cite[Corollary~5.18]{Vinokurov:2025:higher-dim-harm-eigenval} also implies that $im\, u \subset \Sph^n$ for some $\Sph^n \subset \Sph^\infty$.

  In the case when $\partial \Omega \neq \varnothing$, if we additionally assume that $H^1(\Omega)\to L^2(\mu)$ is compact, we automatically obtain $im\, u \subset \Sph^n$ because of the compactness of the operator $T$ from Theorem~\ref{thm:subdiff-eigenval-all}
  and, consequently, finite multiplicity of the eigenvalues. Let $\partial\Omega$ be smooth.
  By considering $w \in H^1(U,\Sph^\infty)$ such that $\supp (w - u|_U) \Subset U \cup \partial\Omega$, as in \cite[Lemma~5.5]{Vinokurov:2025:higher-dim-harm-eigenval},
  and $U$ is from Lemma~\ref{lem:unstable-ring}, one checks that \cite[Lemma~5.5]{Vinokurov:2025:higher-dim-harm-eigenval} generalizes up to the boundary. Hence,
  the harmonic map $u$ is locally energy-minimizing with respect to the free boundary condition $u(\partial\Omega) \subset \Sph^n$.
  By the compactness of $H^1(\Omega)\to L^2(\mu)$ (see \cite[Section~11.9.1]{Mazja:1985:-sobolev-spaces}), we have $\liminf_{r \to 0} r^{2-d}\mu(B_r(p)) = 0$ (we assume $d \geq 2$), and hence
  $u$ is smooth in the interior of $\Omega$ by the standard regularity theory \cite{Hong-Wang:1999:stable-harmonic-maps} (cf. also \cite[Proposition~4.7]{Karpukhin-Stern:2024:harm-map-in-higher-dim}).
  If $\partial \Omega$ is smooth, then $\liminf_{r \to 0} r^{2-d}\mu(B_r(p)) = 0$ even for $p \in \partial \Omega$. As $u$ is a stationary harmonic map into $\Sph^n$ with respect to the free boundary condition
  $u(\partial\Omega) \subset \Sph^n$, the regularity theory from \cite{Scheven:2006:regularity-free-bndry-harm-maps} implies that $u$ is smooth up to the boundary.

\subsubsection{Theorem~\ref{thm:qualitative-stability-stek}}

The proof is very similar to that of Theorem~\ref{thm:qualitative-stability} except
that the analogous variational argument is carried out with measures supported
on the boundary. Accordingly, consider the function
$L^\infty_+(\partial\Omega) \ni \rho \mapsto \overbar{\lambda}_k(\mu + \rho) =\lambda_k(\mu + \rho)(1 + \int_{\partial\Omega} \rho)$.
The subdifferential similarly satisfies
\begin{equation}\label{eq:steklov-subdiff}
   \partial_c(-\overbar{\lambda}_k)(\mu + 0) \subset \lambda_k \cdot\set*{\abs{u}^2\big|_{\partial\Omega} - 1}{u \in H^1(\Omega,\ell^2),\ \Delta u = \lambda u\mu,\ \int_{\partial \Omega} \abs{u}^2 \d\mu \leq 1}.
\end{equation}
Then Proposition~\ref{prop:multip-rule} with $S = L^\infty_+(\partial\Omega)$ yields a map $u \in H^1(\Omega, \ell^2)$ such that $\Delta u = \lambda_k u\mu$,
  $\int_{\partial\Omega} \abs{u}^2d\mu \leq 1$, and
  \begin{equation}
    \int_{\partial\Omega} (\abs{u}^2-1)\rho  \geq 0 \quad \forall \rho \in L^\infty_+(\partial\Omega).
  \end{equation}
Analogously, $\abs{u} \geq 1$ in $L^2(\partial\Omega)$ and $\abs{u} = 1$ in $L^2(\partial\Omega, \mu)$. The maximum principle in the form of~\eqref{ineq:max-principle} can be
deduced for any eigenmap $\Delta u = \lambda_k u \mu$ by approximating $u/\abs{u}$ with $v_\epsilon := u / (\epsilon + \abs{u}^2)^{1/2} \in H^1(\Omega)$.
We may therefore test \eqref{ineq:max-principle} against $(\abs{u} - 1)_+$ to obtain
\begin{equation}
  0\leq \int_{\brc{\abs{u} \geq 1}} \abs{\d\abs{u}}^2 \leq \lambda_k\int_{\partial \Omega} (\abs{u} - 1)_+\abs{u}\d\mu = 0.
\end{equation}
Hence, $\abs{u} \leq 1$ on $\Omega$ and $\abs{u} = 1$ on $\partial\Omega$.

Let $\harm \colon C^0(\partial\Omega) \to C^0(\overbar{\Omega}) \cap C^\infty(\Omega)$ be the harmonic extension operator, and
let $\phi \in \mathrm{Lip}(\partial \Omega)$. Then $\harm (\phi) \in H^1 \cap C^0(\overbar{\Omega})$ and
\begin{multline}
  \lambda_k\int_{\partial\Omega} \phi \d \mu = \lambda_k\int_{\partial\Omega} \brt{u, u \harm(\phi)} \d \mu = \int_{\Omega} \brt{\d u, \d (u \harm (\phi))}
   \\ = \int_{\Omega} \harm (\phi) \abs{\d u}^2 + \frac{1}{2}\int_{\Omega} \brt{\d\harm(\phi), \d\abs{u}^2},
\end{multline}
where the last integral vanishes, since $\abs{u}^2 - 1 \in H^1_0(\Omega)$. Therefore,
\begin{equation}\label{eq:meas-energy}
  \lambda_k\int_{\partial\Omega} \phi \d \mu = \int_{\Omega} \harm (\phi) \abs{\d u}^2 \quad\forall \phi \in \mathrm{Lip}(\partial \Omega)
  \implies \lambda_k\mu = \harm^* (\abs{\d u}^2).
\end{equation}
By Lemma~\ref{lem:harm-ext-adj}, $\mu \in L^1(\partial \Omega)$.
Recall that $\Delta u = \lambda_k u \mu$ in $H^1(\Omega)^*$ means that $\Delta u|_{\Omega} = 0$ and $\partial_n u|_{\partial \Omega} = \lambda_k u \mu \in L^1(\partial\Omega, \ell^2)$,
and we have $\abs{\partial_n u} =\brt{\partial_n u, u} = \lambda_k \mu$. Clearly, $u$ is free boundary harmonic in the sense of \eqref{eq:free-bndry-ball}, since
$\brt{u,v} = 0$ a.e. on $\partial\Omega$ if $v(x) \in T_{u(x)}\mathbb{S}^\infty$ for a.e. $x \in \partial\Omega$.

The regularity in dimension 2 follows from \cite[Lemma~3.2]{Vinokurov:2026:first-two-steklov}. When
$H^1(\Omega)\to L^2(\mu)$ is compact and $\Omega$ is smooth, we obtain $im\, u \subset \mathbb{B}^n$ because of the compactness of the operator $T$ from Theorem~\ref{thm:subdiff-eigenval-all}.
We then proceed, as in the argument for Theorem~\ref{thm:qualitative-stability},
by showing that $u$ is a locally energy-minimizing map into $\R^n$ with respect to the free boundary condition $u(\partial\Omega) \subset \Sph^{n-1}$.
By~\eqref{eq:meas-energy}, we again have $\liminf_{r \to 0} r^{2-d}\int_{B_r(p)}\abs{\d u}^2 \leq C\liminf_{r \to 0} r^{2-d}\mu(B_r(p)) = 0$, and
the regularity theory from \cite{Scheven:2006:regularity-free-bndry-harm-maps} completes the proof of Theorem~\ref{thm:qualitative-stability-stek}.

\begin{remark}\label{rem:free-bndry-l1}
  For a harmonic map $u \in H^1(\Omega, N)$ into a finite-dimensional manifold $N$ with respect to the free boundary condition $u(\partial\Omega) \subset S$, one can choose
  a smooth vector-valued map $F$ defined on a tubular neighborhood of $S \subset N$ such that $F|_S = 0$ and $\d F|_S = P_{(TS)^{\perp}}$,
  where $P_{(TS)^{\perp}}$ is the orthogonal projection onto the normal bundle of $S$ in $N$. Instead of $\abs{u}^2 - 1$, one considers
  $V:= F \circ u \in H^1_0$. Computing
  $\Delta V|_{\Omega} = \mathrm{Hess}^N F(u)(\d u, \d u) \in L^1$, one similarly proves that
  $\partial_n V = \d F (\partial_n u) = \partial_n u\in L^1$.
\end{remark}

\begin{lemma}\label{lem:harm-ext-adj}
  Let $\Omega$ be a bounded Lipschitz domain, and let $\harm \colon C^0(\partial\Omega) \to C^0(\overbar{\Omega})$ be the harmonic extension operator.
  Then its adjoint $\harm^* \colon \mathcal{M}(\overbar{\Omega}) \to \mathcal{M}(\partial\Omega)$ restricts to
  $\harm^*|_{L^1} \colon L^1(\overbar{\Omega}) \to L^1(\partial\Omega)$.
\end{lemma}
\begin{proof}
  Recall that $\harm(\phi)(x) = \int_{\partial \Omega} \phi \d \omega^x$, where $\brc{\omega^x}_{x \in \overbar{\Omega}}$ is the family of harmonic measures.
  Since $\harm \colon C^0(\partial\Omega) \to C^0(\overbar{\Omega})$ and $\harm(1) = 1$, the function $x \mapsto \omega^x(A)$ is Borel for every Borel subset $A \subset \partial\Omega$ and $\omega^x(\partial\Omega) \equiv 1$.
  Furthermore, Dahlberg's theorem (see \cite{Dahlberg:1977:harm-measure} and \cite[Proposition~5.9]{Mitrea-Taylor:1999:bndry-layer-lip}) states that for Lipschitz $\Omega$,
  the harmonic measures $\brc{\omega^x}_{x \in \Omega}$ are absolutely continuous with respect to $\mathcal{H}^{d-1}\big|_{\partial\Omega}$.

  For a function $f \in L^1(\Omega)$, the measure $\nu = \harm^*(f) \in \mathcal{M}(\partial\Omega)$ is given by
  \begin{equation}
    \nu(A) = \int_{\Omega} f(x) \omega^x(A).
  \end{equation}
  If $\mathcal{H}^{d-1}(A) = 0$, Dahlberg's theorem yields $\omega^x(A) = 0$ and $\abs{\nu}(A) = 0$. Hence $\nu \in L^1(\partial\Omega)$ by
  the Radon--Nikodym theorem.
\end{proof}

\subsubsection{Corollary~\ref{cor:existence}}
  Let $\brc{\mu_n}$ be a maximizing sequence of continuous probability measures such that $\lambda_k(\mu_n) \to \Lambda_k(\Omega,g)$. After passing to a subsequence
  if necessary, we may assume $\mu_n \oset{w^*}{\to} \mu \in \mathcal{M}_+(\overbar{\Omega})$ with $\mu(\overbar{\Omega})=1$. Then
  \begin{equation}\label{eq:upper-semi-c}
    \limsup_n \lambda_k(\mu_n) \leq \lambda_k(\mu)
  \end{equation}
  by~\cite[Proposition~1.1]{Kokarev:2014:measure-eigenval}. The limiting measure $\mu$ decomposes as the sum of
  its continuous and atomic parts:
  \begin{equation}
    \mu = \mu^c + \sum_{p} w_p \delta_p.
  \end{equation}
  We extend each measure to $M$ ($\mu(A):= \mu(A\cap \overbar{\Omega})$) and analyze their behavior near the atoms
  of $\mu$. If $d \geq 3$, \cite[Proposition~4.6]{Vinokurov:2025:higher-dim-harm-eigenval} implies that $\mu = \mu^c$,
  and $\mu$ is a maximizing measure by \eqref{eq:upper-semi-c}. The equality $\Lambda_k(\Omega,g) = \Lambda_k^*(\Omega,g)$ follows from Theorem~\ref{thm:qualitative-stability}.

  If $d = 2$, \cite[Proposition~4.10]{Vinokurov:2026:p-harm-and-conf-class-opt} implies that for any sufficiently small $\epsilon > 0$, one has
  \begin{equation}\label{eq:meas-upper-bound}
    (1-\epsilon r) \Lambda_k(\Omega,g)\leq \overbar{\lambda}_{k_0}(\Omega,\mu^c) + \sum_{i\geq 1} \Lambda_{k_i}(\Sph^2)
  \end{equation}
  for some $0\leq r \leq k$, $\sum_{i\geq 0} k_i = k - r$ (here, $r$ corresponds to the number of secondary bubbles at infinity). On the other hand,
  \begin{equation}\label{eq:meas-lower-bound}
    \max_{1\leq b \leq k} \brc{\Lambda_{k-b}^*(\Omega,[g]) + 8\pi b} \leq \Lambda_k^*(\Omega,[g]) \leq \Lambda_k(\Omega,[g]),
  \end{equation}
  as follows from \cite{Colbois-ElSoufi:2003:extremal}, where
  $8\pi b = \Lambda_b^*(\Sph^2)$ by \cite[Theorem~1.2]{Karpukhin-Nadirashvili-Penskoi-Polterovich:2021:eigenval-on-sphere}.
  Applying \eqref{eq:meas-upper-bound}, \eqref{eq:meas-lower-bound} to $\Omega = \Sph^2$
  and using Theorem~\ref{thm:qualitative-stability}, we obtain
  by induction on $k$ that $r=0$ and $\Lambda_k(\Sph^2) = \Lambda_k^*(\Sph^2) = 8\pi k$ (cf. \cite[(4.17) and below]{Vinokurov:2026:p-harm-and-conf-class-opt}).

  Therefore, $\Lambda_k(\Omega,g)\leq \overbar{\lambda}_{k-b}(\Omega,\mu^c) + 8\pi b$ for some $b \leq k$. Combining this inequality with the assumption for $d=2$, we
  deduce that $b = 0$, and $\mu^c$ is a maximizing measure. The equality $\Lambda_k(\Omega,g) = \Lambda_k^*(\Omega,g)$ then follows
  from Theorem~\ref{thm:qualitative-stability} by induction on $k$. Thus, the proof is complete.

  \begin{remark}\label{rem:superlevel-stability}
    The upper semicontinuity of $\mu \mapsto \lambda_k(\mu)$ and the absence of
    atoms in weak-$*$ limits for $d\geq3$ imply that the superlevel sets $\set{\mu \in \mathcal{P}^c(\overbar{\Omega})}{\lambda_k(\mu) \geq c}$
    are weak-$*$ stable for every $c > 0$, where
    $\mu \mapsto \lambda_k(\mu)$ is regarded as a function on the space of continuous probability measures
    $\mathcal{P}^c(\overbar{\Omega}):= \set{\mu \in \mathcal{M}_+^c(\overbar{\Omega})}{\mu(\overbar{\Omega})=1}$. In dimension $2$, the weak-$*$ stability
    holds at least for the superlevel sets $\set{\mu \in \mathcal{P}^c(\overbar{\Omega})}{\lambda_k(\mu) \geq c}$
    for every $c > \Lambda_{k-1}(\Omega,[g]) + 8\pi$.
  \end{remark}

\subsubsection{Corollary~\ref{cor:existence-stek}}
  We proceed as in the proof of Corollary~\ref{cor:existence}. In dimensions $d\geq 3$,  Remark~\ref{rem:superlevel-stability}
  shows that if $\lambda_k(\mu_n) \to \Sigma_k(\Omega,g)$ for a sequence $\brc{\mu_n} \subset \mathcal{P}^c(\partial\Omega)$, then, up to a subsequence,
  $\mu_n \oset{w^*}{\to} \mu \in \mathcal{P}^c(\partial\Omega)$ with $\lambda_k(\mu) = \Sigma_k(\Omega,g)$. Note that the test functions constructed by
  \cite[Proposition~4.6]{Vinokurov:2025:higher-dim-harm-eigenval} are, a priori, only Lipschitz, as permitted by the assumed regularity of $\Omega$.
  Nevertheless, the argument still applies.

  The case $d = 2$ is handled in the same way as in the proof of Corollary~\ref{cor:existence}. By \cite[Proposition~4.10]{Vinokurov:2026:p-harm-and-conf-class-opt}
  and \cite[Corollary~A.8]{Vinokurov:2025:sym-eigen-val-lms}, for every sufficiently small $\epsilon > 0$, we have
  \begin{equation}
    (1-\epsilon r) \Sigma_k(\Omega,g)\leq \overbar{\lambda}_{k_0}(\Omega,\mu^c) + \sum_{i\geq 1} \Sigma_{k_i}(\mathbb{D}^2),
  \end{equation}
  since $\R^2_+ \cup \brc{\infty}$, $\mathbb{S}^2_+$, $\mathbb{D}^2$ are conformally equivalent. Now,
  \begin{equation}
    \max_{1\leq b \leq k} \brc{\Sigma_{k-b}^*(\Omega,[g]) + 2\pi b} \leq \Sigma_k^*(\Omega,[g]) \leq \Sigma_k(\Omega,[g])
  \end{equation}
  by \cite{Fraser-Schoen:2020:steklov-unions} (see also \cite[Section~3.3.1]{Vinokurov:2025:sym-eigen-val-lms}),
  where $2\pi b = \Sigma_b^*(\mathbb{D}^2)$ (see \cite[Remark 1.12]{Vinokurov:2025:sym-eigen-val-lms}).
  The remainder of the argument is identical to that of Corollary~\ref{cor:existence}.

\subsubsection{Corollary~\ref{cor:nonessen-boundary-dim-2}}
  To prove $\overbar{\lambda}_k(\Omega, \mu) < \Lambda_k(\Omega, g) \leq \Lambda_k(M, g)$, extend the measure $\mu$ to all of $M$, $\tilde{\mu}(A):= \mu(A\cap \overbar{\Omega})$.
  From the definition of $\overbar{\lambda}_k$, we have $\overbar{\lambda}_k(\Omega,\mu) \leq \overbar{\lambda}_k(M, \tilde{\mu})\leq \Lambda_k(M, g)$.
  If $\overbar{\lambda}_k(\Omega,\mu) = \Lambda_k(\Omega, g)$, then
  Theorem~\ref{thm:qualitative-stability} implies that $0 \neq \lambda_k\mu = \abs{\d u}^2$ for a harmonic map $u \in C^\infty(\Omega\setminus\sing u, \Sph^\infty)$.
  The open set $\Omega\setminus\sing u$ is connected and $H^1(\Omega\setminus\sing u) = H^1(\Omega)$, as the closed set $\sing u \subset \Omega$ has vanishing $(d-2)$-Hausdorff measure.

  The differential of $u$, $\lambda_k\mu = \abs{\d u}^2$, vanishes on the nonempty open set $U := \Omega\setminus(\sing u \cup \supp \mu)$, and hence
  $\Delta u = \abs{\d u}^2 u = 0$ on $U$ as well. Let $u = c \in \ell^2$ on a connected component of $U$.
  By the unique continuation principle applied to the functions $\brt{u,b}$ with $b\in \ell^2$, $b \perp c$, we conclude
  that $u = c$ on the domain $\Omega\setminus\sing u$ and hence on $\Omega$.
  This contradiction implies that we must have $\overbar{\lambda}_k(\Omega,\mu) < \Lambda_k(\Omega, g)$. Taking
  $\Omega = M$ with $\tilde{\mu}$ we obtain that $\overbar{\lambda}_k(\Omega,\mu) \leq \overbar{\lambda}_k(M, \tilde{\mu}) < \Lambda_k(M, g)$ if
  $\Omega \neq M$.

  Now, suppose $d = \dim \Omega = 2$ and $\Omega^* \approx \Sph^2$. From the argument above, we already have that $\Lambda_k(\Omega, [g])\leq \Lambda_k(\Omega^*, [g]) = 8\pi k$,
  and the lower bound follows from Remark~\ref{rem:lower-lambda-bounds}.

  If $\Omega^* \not\approx \Sph^2$, let $k \geq 1$ be the smallest positive integer such that $\Lambda_k(\Omega, [g]) = \Lambda_k(\Omega^*, [g])$.
  Then either $k = 1$, or $\Lambda_{k-1}(\Omega, [g]) + 8\pi < \Lambda_k(\Omega, [g])$. Otherwise, we would have
  \begin{equation}
    \Lambda_k(\Omega^*, [g]) = \Lambda_k(\Omega, [g]) = \Lambda_{k-1}(\Omega, [g]) + 8\pi \leq \Lambda_{k-1}(\Omega^*, [g]) + 8\pi \leq \Lambda_k(\Omega^*, [g]);
  \end{equation}
  hence $\Lambda_{k-1}(\Omega, [g]) = \Lambda_{k-1}(\Omega^*, [g])$, contradicting the minimality of $k$. On the other hand,
  by Remark~\ref{rem:lambda-1}, one has $8\pi < \Lambda_1(\Omega^*, [g]) = \Lambda_1(\Omega, [g])$ if $k = 1$.

  Thus, we have established that
  $\Lambda_{k-1}(\Omega, [g]) + 8\pi < \Lambda_k(\Omega, [g])$ regardless of whether $k = 1$. Then Corollary~\ref{cor:existence} applies, leading to the existence
  of a measure $\mu \in \mathcal{M}_{+}^c(\overbar{\Omega})$ such that $\overbar{\lambda}_k(\Omega,\mu) = \Lambda_k(\Omega, [g]) = \Lambda_k(\Omega^*, [g])$.
  Since this is impossible, as shown at the beginning of the proof, we must have $\Lambda_k(\Omega, [g]) < \Lambda_k(\Omega^*, [g])$ for all $k \geq 1$
  when $\Omega^* \not\approx \Sph^2$.

\addsec{Acknowledgements}
This work forms part of the author's PhD thesis, written under the supervision of Mikhail Karpukhin and Iosif Polterovich.
The author is grateful to them for their guidance and numerous fruitful discussions.

The author would like to thank the Isaac Newton Institute for Mathematical Sciences, Cambridge,
for its support and hospitality during the programme \emph{Geometric spectral theory and applications},
where part of the work on this paper was undertaken. This work was supported by EPSRC grant EP/Z000580/1.

The author acknowledges the support of the Natural Sciences and Engineering Research Council of Canada (NSERC)
through the Canada Graduate Research Scholarship--Doctoral (CGRS D).

\printbibliography

\end{document}